\documentclass[a4paper]{article}

\usepackage{amssymb,amsmath,amsthm}
\usepackage[english]{babel}

\usepackage{todonotes}
\usepackage{verbatim} 
\usepackage{cite}
\usepackage{enumerate}

\usepackage{color}
\definecolor{red}{RGB}{255,0,0}

\newcommand{\R}{\mathbb{R}}

\newcommand{\E}{\mathbb{E}}
\newcommand{\N}{\mathbb{N}}
\newcommand{\Z}{\mathbb{Z}}
\newcommand{\PR}{\mathbb{P}}

\newcommand{\B}{\mathcal{B}}
\newcommand{\1}{\textbf{1}}
\newcommand{\dd}{ \mathrm{d}}

\newcommand{\vn}[1]{\left| \! \left| #1\right| \! \right|} 
\newcommand{\tn}[1]{\left| \! \left| \! \left|#1\right| \! \right| \! \right|}

\numberwithin{equation}{section}

\newtheorem{Def}{Definition}[section]
\newtheorem{The}[Def]{Theorem}
\newtheorem{Lem}[Def]{Lemma}
\newtheorem{Cor}[Def]{Corollary}
\newtheorem{Prop}[Def]{Proposition}

\newtheorem{Ass}[Def]{Assumption}

\title{article about Masterthesis}
\author{Richard Kraaij}

\begin{document}

\author{Richard Kraaij}

\begin{center}
{\Large Stationary product measures for conservative particle systems and ergodicity criteria}
\bigskip

Richard Kraaij\footnote{Delft Institute of Applied Mathematics,
Technical University Delft,
Mekelweg 4, 2628 CD Delft, the Netherlands,
 E-mail \texttt{r.c.kraaij@tudelft.nl}}

\bigskip

\today
\end{center}

\begin{abstract}
We study conservative particle systems on $W^S$, where $S$ is countable and $W = \{0, \dots, N\}$ or $W = \N$, where the generator reads
\begin{equation*}
Lf(\eta) = \sum_{x,y} p(x,y) b(\eta_x,\eta_y) (f(\eta - \delta_x + \delta_y) - f(\eta)).
\end{equation*}
Under assumptions on $b$ and the assumption that $p$ is finite range, which allow for the exclusion, zero range and misanthrope processes, we show exactly what the stationary product measures are. 

Furthermore we show that a stationary measure $\mu$ is ergodic if and only if the tail sigma algebra of the partial sums is trivial under $\mu$. This is a consequence of a more general result on interacting particle systems that shows that a stationary measure is ergodic if and only if the sigma algebra of sets invariant under the transformations of the process is trivial. We apply this result combined with a coupling argument on the stationary product measures to determine which product measures are ergodic. For the case that $W$ is finite this gives a complete characterisation.

For the case that $W = \N$ we show that for nearly all functions $b$ a stationary product measure is ergodic if and only if it is supported by configurations with an infinite amount of particles. We show that this picture is not complete, we give an example of a system where $b$ is such that there is a stationary product measure which is not ergodic, even though it concentrates on configurations with an infinite number of particles.

\end{abstract}

\section{Introduction}

For the exclusion, inclusion, zero range and misanthrope process \cite{BookLiggett1,CocozzaMisanthrope,AndjelZR, GroRedVafInc} there is a long history of research into the stationary and ergodic measures. For  the exclusion process it is known for a long time that the model has invariant product measures which are indexed by the particle density per site. It was shown that the model has stationary measures which have a constant density and that there are measures which are indexed with a parameter $(\lambda_x)_{x \in S}$ that is reversible with respect to the random walk kernel $p$, i.e. $\lambda_x p(x,y) = \lambda_y p(y,x)$, see e.g. \cite{BookLiggett1}. 

This picture was shown to be true for other models as well \cite{CocozzaMisanthrope, GroRedVafInc}. For the zero range process however this picture was not complete as was shown in Andjel \cite{AndjelZR}. The underlying parameters $\lambda$ for a product measure in case of the zero range process were only required to satisfy $\sum_x \lambda_x p(x,y) = \lambda_y$. In 2005 Bramson and Liggett \cite{LiggettBramson} extended the picture for the exclusion process by showing that product measures for which $(\lambda_x)_{x \in S}$ satisfies $\sum_x \lambda_x p(x,y) = \lambda_y$ and if $\lambda_x p(x,y) \neq \lambda_y p(y,x)$ then $\lambda_x = \lambda_y$ are stationary as well.

\bigskip

The problem of finding \textit{all ergodic measures} for such systems is still open. Progress has been made to classify which stationary product measures are ergodic. For the exclusion process this process was solved in Jung \cite{Jung}, for the zero range process the problem is solved in Sethuraman \cite{SSEx} with additional conditions on the interaction function $g$. Also for the misanthrope process and the inclusion process ergodicity problems regarding product measures are solved \cite{AnCoRoMisanthrope, GroRedVafInc}. For different models different methods are being used, Sethuraman \cite{SSEx} however, uses an approach that works for a range of models. 

\bigskip

In this paper show in  that these questions can dealt with regardless of the specific model, i.e. we will work with systems with a generator of the form $Lf(\eta) = \sum_{x,y} p(x,y) b(\eta_x,\eta_y) (f(\eta - \delta_x + \delta_y) - f(\eta))$ where $p$ is finite range and where function $b$ depends on the model that we are working with. We will take $b$ bounded for convenience, but the methods are not restricted to this case.

\bigskip

We start in section \ref{sec:proofergo} by proving that a stationary measure $\mu$ is ergodic if and only if the tail sigma algebra of the partial sums is trivial under $\mu$. In fact this is a consequence of a result that is valid for more general interacting particle systems(IPS). We will show that a stationary measure $\mu$ for a general IPS is ergodic if and only if the sigma algebra of sets that are invariant under the possible transformations of the system is trivial under $\mu$. This result also shows that stationary measures for Glauber dynamics are ergodic if and only if they are tail trivial.

\bigskip

In section \ref{sect:product} we will address the question of stationarity of product measures. We show that the idea of Bramson and Liggett \cite{LiggettBramson} extends to other models and that the structure of the set of stationary invariant measures depends crucially on the structure of the function $b$. If you put more restrictions on $b$ less parameter sets $\lambda$ will yield a stationary product measure. Also we show that these are exactly the stationary product measures of this type and no more can be found.

\bigskip

After that we apply these results in section \ref{sect:proof} to show which product measures are ergodic. We use a coupling proof to extend the results of Jung \cite{Jung} to the case where $W = \{0, \dots, N\}$, hence completely resolving the question if $W$ is finite. We use the same techniques to show similar results for the case that $W = \N$. In this case however we find some interesting behaviour. For most functions $b$ we see that a product measure is ergodic if it has zero mass on configurations with a finite number of particles, this behaviour is consistent with the behaviour found for the zero range process \cite{SSEx}. For certain functions however this behaviour breaks down, as we illustrate in section \ref{sec:deterministicprofile} with  an example of a system where we have a stationary, but non ergodic, product measure which concentrates on configurations with an infinite amount of particles but which follows a certain increasing deterministic profile.

\subsection{Main Results} \label{sec:notanddef}

Let $E = W^S$ be the set of configurations $(\eta_x)_{x \in S}$ for a countable set $W$ and a countable set $S$. Let $\mathcal{B}$ be the product $\sigma$-algebra. For example the exclusion process is defined on $\{0,1\}^S$, the zero range process on $\N^S$ and the stochastic Ising model on $\{-1,1\}^S$. By $\eta(t) = (\eta(t))_{i \in S}$ we describe the configuration of the process at time $t$. We order $S$ by a bijection $\phi : S \rightarrow \N$, so $i < j$ if $\phi(i) < \phi(j)$. Using this ordering define $S_n = \{x \in S \; : \; \phi(x) \leq n \}$ and $\mathcal{B}_n = \sigma \{\eta_i \; : \; i \in S_n \}$.
Define
\begin{equation*}
\Delta_f(i) = \sup \left\{|f(\eta) - f(\zeta)| \; : \; \text{ for } j \neq i: \; \eta_j = \zeta_j \right\}
\end{equation*}
the variation of $f$ at $i \in S$. Define the space of test functions by
\begin{equation}
D = \left\{f \in C_b(E) \; : \; \tn{f} := \sum_{x \in S} \Delta_f(x) < \infty \right\}. \label{eqn:defD}
\end{equation}
Define $\eta^{x,y} = \eta - \delta_x + \delta_y$ and let $\nabla_{x,y}f(\eta) = f(\eta^{x,y}) - f(\eta)$
For $f \in D$ we define 
\begin{equation*}
L^{b,p}f(\eta) = \sum_{x,y} p(x,y) b(\eta_x,\eta_y) \nabla_{x,y}f(\eta).
\end{equation*}
Note that for $W = \{0,1\}$ and $b(n,k) = n(1-k)$ we obtain the exclusion process and that for $W = \N$ and $b(n,k) = g(n)$ we obtain the zero range process. We will refer to $b$ as the rate function and to this class of processes by the name product type processes.
We will assume that
\begin{Ass} \label{Ass:Cp}
$p$ is finite range and
\begin{equation*}
\sup_x \sum_y(p(x,y)+p(y,x)) = C_p < \infty.
\end{equation*}
\end{Ass}

We also make the following assumption.

\begin{Ass} \label{Ass:irreducible}
$p$ is irreducible and $b$ is positive except for the two cases $b(0,\cdot) = 0$ and if $W = \{0, \dots, N\}$: $b(\cdot,N) = 0$
\end{Ass}

In the case that $W$ is a finite set we know by theorem I.3.9 in Liggett\cite{BookLiggett1} that there exists a process $\eta(t)$ and semigroup $S_t :C(E) \rightarrow C(E)$ corresponding to $L^{b,p}$. With the same techniques it is not hard to show that there is a process $\eta(t)$ and semigroup $S_t : \overline{D} \rightarrow \overline{D}$ in the case that $W = \N$ and $b$ is bounded. In both cases $D$ is a core for $L^{b,p}$. Note that in the case that $W=\N$ it is not the case that $\overline{D} = C(E)$. It seems that $\overline{D}$ which is the uniform closure of bounded local functions is the natural space to work with. 

The zero range process has been constructed also for unbounded $b$ in Andjel \cite{AndjelZR}, the results that we obtain in this article do not improve upon the results of Sethuraman \cite{SSEx} with respect to the zero range process, so we will not deal with with this construction. The methods developed here apply to the zero range process as the methods are valid regardless of the structure of $b$. 

\bigskip

For a more general interacting particle system we follow the notation of Liggett \cite{BookLiggett1}. For $T$ a finite subset of $S$ and $\zeta \in W^T$ let $c_T(\eta,\zeta)$ be the rate at which the system makes a transformation from configuration $\eta$ to configuration $\theta_{T,\zeta}(\eta)$ which is defined by
\begin{equation*}
\theta_{T,\zeta}(\eta)_i =
  \begin{cases}
   \eta_i   & \text{if } i \notin T \\
   \zeta_i  & \text{if } i \in T
  \end{cases}
\end{equation*}
and put $c_T = \sup\{c_T(\eta,W^T) \; : \; \eta \in E\}$. Lastly define $\nabla_{T,\zeta}f(\eta) = f(\theta_{T,\zeta}(\eta)) - f(\eta)$. 

For functions $f \in D$ define $L$ to be
\begin{equation} \label{eqn:defL}
Lf(\eta) = \sum_T \int c_T(\eta,\dd \zeta) \nabla_{T,\zeta}f(\eta).
\end{equation}
If we assume $W$ to be finite theorem I.3.9 in Liggett \cite{BookLiggett1} gives that $L$ generates a Markov process $\eta(t)$ and semigroup $S_t : C(E) \rightarrow C(E)$ for which $D$ is a core. One of the two assumptions for this theorem to hold is
\begin{Ass} \label{Ass:C}
\begin{equation*}
\sup_x \sum_{T \ni x} c_T = C < \infty
\end{equation*}
\end{Ass}
we state this assumption as we need it for calculations later on. Note that in the particular case of product type systems assumption \ref{Ass:Cp} implies assumption \ref{Ass:C}.

\bigskip

Furthermore we define the set of stationary measures for the process generated by $L$ by $\mathcal{I}(L)$, proposition 4.9.2 in Ethier and Kurtz \cite{EthierKurtz} shows that
\begin{equation} \label{eqn:I}
\mathcal{I}(L) = \left\{\mu \; : \; \int Lf \dd \mu =0 \quad \forall \; f \in D \right\}.
\end{equation}

We start with stating the result on ergodicity. 

\subsection{Ergodic measures for general IPS} \label{sect:ergodicmeasuresgeneral}

In this section we work with a generator $L$ that is given by equation (\ref{eqn:defL}).
For the results that follow we need the following assumption.

\begin{Ass} \label{ass:wederkerigheid}
For $\eta \in E$, $T \subset S$ a finite set, $\zeta \in W^T$ such that $c_T(\eta,\zeta) > 0$ there is a $n \in \N$, there are finite sets $T_1, \dots, T_n \subset S$ and $\zeta_1 \in W^{T_1}, \dots, \zeta_n \in W^{T_n}$ such that for all $i \leq n$:
\begin{equation*}
c_{T_i}\left(\theta_{T_{i-1},\zeta_{i-1}} \circ \dots \circ \theta_{T_{1},\zeta_{1}} \circ \theta_{T,\zeta} (\eta), \zeta_i \right) > 0
\end{equation*}
and
\begin{equation*}
\theta_{T_{n},\zeta_{n}} \circ \dots \circ \theta_{T_{1},\zeta_{1}} \circ \theta_{T,\zeta} (\eta) = \eta
\end{equation*}
\end{Ass}
This assumptions states that if the Markov process allows the transformation from $\eta$ to $\theta_{T,\zeta}(\eta)$ then there is a sequence of possible transformations that returns the configuration to $\eta$. Under this assumption we can define the following $\sigma$-algebra.

\begin{Def} \label{Def:GL}
For a generator $L$ define the $\sigma$-algebra $\mathcal{G}_L$ of sets that are invariant under transformations of the process generated by $L$. That means that if $G \in \mathcal{G}_L$ and $\eta \in G$, $T \subset S$ finite, $\zeta \in W^T$ such that $c_T(\eta,\zeta) > 0$ then $\theta_{T,\zeta}(\eta) \in G$.
\end{Def}
Note that by assumption \ref{ass:wederkerigheid} $\mathcal{G}_L$ is a $\sigma$-algebra. We now state the main theorem of this section.

\begin{The} \label{the:ergoprime}
If $L$ generates a Markov process and $\mu \in \mathcal{I}(L)$, then $\mu$ is ergodic if and only if $\mathcal{G}_L$ is trivial under $\mu$.
\end{The}

We give two corollaries to this theorem regarding two examples, see corollary \ref{Cor:ergodicity} below.
The first class of examples are spin flip systems, with a generators that read
\begin{equation*}
Lf(\eta) = \sum_x r(x,\eta) (f(\eta^x) - f(\eta)) 
\end{equation*}
for some rate function $r$, where $W = \{-1,1\}$ and $\eta^x_y = \eta_y$ if $y \neq x$ and $\eta^x_x = - \eta_x$. Important examples are stochastic Ising models.

The second class of examples are conservative systems, of which the product type systems are a special case. Also Kawasaki dynamics belongs to this case.
\begin{equation*}
Lf(\eta) = \sum_{x,y} r(x,y,\eta) \nabla_{x,y} f(\eta)
\end{equation*}

\begin{Def}
We define the following $\sigma$-algebras.
\begin{enumerate}[(a)]
\item The tail $\sigma$-algebra: 
\begin{equation*}
\mathcal{T} = \bigcap_n \sigma \; \left(\eta_x \; : \; x \in S \text{ such that } \phi(x) \geq n \right)
\end{equation*} 
\item The tail $\sigma$-algebra of the partial sums: 
\begin{equation*}
\mathcal{H} = \bigcap_n \; \mathcal{H}_n = \bigcap_n \; \sigma \left(\sum_{\phi(x) \leq m} \eta_x \; : \; m > n \right)
\end{equation*}
\item Let $\mathcal{A}$ be the $\sigma$-algebra of events that are invariant under moving particles from one site to another.
\end{enumerate}
\end{Def}

First we show that the last two $\sigma$-algebras are equal
\begin{Lem} \label{Lem:AH}
It holds that $\mathcal{A} = \mathcal{H}$.
\end{Lem}

We use this information combined with the following irreducibility assumptions to obtain corollary \ref{Cor:ergodicity}.
\begin{Ass} \label{Ass:cirreducible}
In the case that we are working with a conservative particle system, we assume that $c$ is irreducible. This means that if we have two configurations $\eta$ and $\hat{\eta}$ such that there is a finite box $B \subset S$ such that $\eta$ agrees with $\hat{\eta}$ outside $B$ and $\sum_{x \in B} \eta_x = \sum_{x \in B} \hat{\eta}_x$, then there exists a sequence of configurations $\eta = \eta_0 , \dots , \eta_n = \hat{\eta}$, so that we have a sequence of sites in $S$: $x_0, \dots x_n$ such that $\eta_i = \eta_{i-1}^{x_{i-1},x_i}$ and jump rate $r(x_{i-1},x_i,\eta_{i-1}) > 0$.
\end{Ass}
Look for example at a product type conservative particle system, then this assumption is satisfied as a consequence of assumption \ref{Ass:irreducible}. It is easy to see that this assumption implies that $\mathcal{G}_L = \mathcal{A} = \mathcal{H}$.

\begin{Ass} \label{Ass:cirreducible2}
In the case that we are working with a spin flip system we assume that
\begin{equation*}
\inf_{x,\eta} r(x,\eta) > 0
\end{equation*}
\end{Ass}
Under this assumption we see that $\mathcal{G}_L = \mathcal{T}$. Note that we it is possible to work with a more general assumption then \ref{Ass:cirreducible2} but we do not need that here.

\begin{Cor} \label{Cor:ergodicity}
\begin{enumerate}[(a)]
\item If $L$ generates a spin flip system and $\mu \in \mathcal{I}(L)$, then $\mathcal{G}_L = \mathcal{T}$, hence $\mu$ is ergodic if and only if $\mathcal{T}$ is trivial under $\mu$.
\item If $L$ generates a conservative particle system and $\mu \in \mathcal{I}(L)$ then $\mathcal{G}_L = \mathcal{A} = \mathcal{H}$, hence $\mu$ is ergodic if and only if $\mathcal{H}$ is trivial under $\mu$.
\end{enumerate}
\end{Cor}

The use of theorem \ref{the:ergoprime} is not restricted to these cases however. For example it can also be applied to the tagged particle proces \cite{BookLiggett2, ArtSaadaZRtaggedparticle, ArtSaadaLTSEP}. These models are just like the product type IPS, but now one is interested in the properties of a single particle, the tagged particle. One starts the dynamics from a translation invariant stationary product measure. Important information can be obtained by looking at the environment as seen from this tagged particle: the environment process. It is proven that the environment process also has a stationary product measure, see e.g. \cite{BookLiggett2}, proposition III.4.3, or \cite{ArtSaadaZRtaggedparticle}, proposition 7. One would like to prove that this measure is ergodic, see \cite{BookLiggett2}, proposition III.4.8. The results in this paper give with minor adaptations a shorter proof of this proposition. First of all a stationary measure $\nu$ is ergodic if and only if $\mathcal{A} \cap \mathcal{I}$ is trivial under $\nu$, where $\mathcal{I}$ is the $\sigma$-algebra of shift invariant sets. The results below in theorem \ref{The:ergodicproduct} show under which conditions $\mathcal{A}$, hence $\mathcal{A}\cap \mathcal{I}$ is trivial under $\mu$.

\subsection{Results on product measures for product type conservative particle systems} \label{sect:productmeasures}

We return to product type systems where the generator reads
\begin{equation*}
L^{b,p}f(\eta) = \sum_{x,y} p(x,y) b(\eta_x,\eta_y) \nabla_{x,y} f(\eta).
\end{equation*}
For the existence of product stationary measures we make the following two assumptions
\begin{Ass} \label{Ass:1}
For all $i,j \in W$  we have
\begin{equation*}
\frac{b(i+1,j-1)}{b(j,i)}= \frac{b(1,j-1)}{b(j,0)}\frac{b(i+1,0)}{b(1,i)}
\end{equation*}
\end{Ass}

This property ensures that we obtain a set of invariant product measures and can be traced back to Cocozza-Thivent \cite{CocozzaMisanthrope}.
The second assumption is needed for the case $W = \N$ and will be explained below.

\begin{Ass} \label{Ass:2}
If $W = \N$ we assume that
\begin{equation*}
\inf_i \frac{b(i+1,0)}{b(1,i)} = I > 0.
\end{equation*}
\end{Ass}

Under these assumptions the process generated by $L^{b,p}$ has a natural class of invariant product measures. These are defined in the following way. 

\begin{equation} \label{eqn:defa}
\begin{aligned}
a_0 & = 1 \\
a_k & = \prod_{i=0}^{k-1} \frac{b(1,i)}{b(i+1,0)}  \\
Z_\lambda & = \sum_k  a_k \lambda^k \\
\lambda^* & = \liminf_j \frac{b(j+1,0)}{b(1,j)} \\
\end{aligned}
\end{equation}
$\lambda^*$ is the radius of convergence of the formal sum $Z_\lambda$, so for $\lambda < \lambda^*$ we have that $Z_\lambda < \infty$. Note that if we are working with $W = \{0,\dots,N\}$ then we only define $a_k$ for $k \leq N$, hence $\lambda^*$ will be infinite. Because of assumption \ref{Ass:2} we know that $\lambda^* > 0$.

\bigskip

Extend $\lambda$ to have one value for each point in $S$, so $\lambda \in [0,\lambda^*)^S$. Then:

\begin{Def}
The one site marginal: let $x \in S$, then
\begin{equation*}
\mu_{\lambda_x}(n) = Z_{\lambda_x}^{-1} a_n \lambda_x^n 
\end{equation*}
The measure $\mu_\lambda$ will be the product measure on $W^S$:
\begin{equation*}
\mu_\lambda = \otimes_{x \in S} \mu_{\lambda_x}
\end{equation*}
The set of measures of this type is denoted by
\begin{equation*}
\mathcal{P}_\otimes(b) = \left\{\mu_\lambda \; : \; \lambda \in [0,\lambda^*)^S \right\}.
\end{equation*}
\end{Def}

We see that given a function $b$ we obtain the set of measures $\mathcal{P}_\otimes(b)$. Note however that different $b$'s can lead to the same set of probability measures.
We identify the stationary product measures of this type.

\begin{Prop} \label{prop:invariantmeasures}
Let $\lambda$ be a solution of $\sum_x \lambda_x p(x,y) = \lambda_y \sum_x  p(y,x)$.
Depending on the structure of $b$ we have the following:
\begin{enumerate}[(a)]
\item If for all $k$ it holds that $b(n,k) = b(n,0)$, i.e. the zero range process, then $\mu_\lambda \in \mathcal{I}(L^{b,p})$.
\item If $b(n,k) - b(k,n) = b(n,0) - b(k,0)$ and $\lambda$ is such that if $\lambda_x p(x,y) \neq \lambda_y p(y,x)$, then $\lambda_x = \lambda_y$ , then it holds that $\mu_\lambda \in \mathcal{I}(L^{b,p})$.
\item If $\lambda_x p(x,y) = \lambda_y p(y,x)$ for all $x$ and $y$ then $\mu_\lambda \in \mathcal{I}(L^{b,p})$.
\end{enumerate}
Furthermore, an invariant measure $\mu_\lambda$ in the set $\mathcal{P}_\otimes(b)$ must be of one of these three types, i.e. $\lambda$ is a solution of $\sum_x \lambda_x p(x,y) = \lambda_y \sum_x  p(y,x)$ and the pair $(\lambda, b)$ satisfies (a), (b) or (c).
\end{Prop}

\textit{remark 1} Note that the condition $b(n,k) - b(k,n) = b(n,0) - b(k,0)$ is equivalent to $b(n,k) = r(n) + s(n,k)$ where $s$ is symmetric. Choose for example $r(n) = b(n,0)$. 

\textit{2} Furthermore it is an interesting question whether these results can be extended to infinite range $p$. 

\bigskip

Now that we know what the class of invariant product measures is, we can apply corollary \ref{Cor:ergodicity}. A coupling argument is used to prove the following theorem. For a fixed generator $L^{b,p}$ pick $\mu_\lambda \in \mathcal{I}(L^{b,p}) \cap \mathcal{P}_\otimes(b)$.

\begin{The} \label{The:ergodicproduct} 
\begin{enumerate}[(a)]
\item Suppose $W = \{0,\dots,N \}$ then $\mu_\lambda$ is ergodic if and only if $  \sum_{i : \lambda_i < 1} \lambda_i + \sum_{i: \lambda_i \geq 1} \frac{1}{\lambda_i} = \infty$. 

\item If $W = \N$ and $(\ast \ast)$: $\lambda^*<\infty$ or  $\lambda^* = \infty$ and there is a finite set $D$, such that $\gcd(D) = 1$ and
\begin{equation*}
D \subset \left\{d \geq 1 \; : \; \sup_k \frac{a_k^2}{a_{k-d}a_{k+d}} = \sup_k \prod_{i=0}^{d-1}\frac{b(k+i+1,k-i-1)}{b(k-i,k+i)} < \infty \right\}
\end{equation*}
then it holds that $\mu_\lambda$ is ergodic if and only if  $\sum_i \lambda_i = \infty$. 

\bigskip

\item Furthermore if $W = \N$ and $ \sum_{i : \lambda_i < 1} \lambda_i +  \sum_{i: \lambda_i \geq 1} \frac{1}{\lambda_i} = \infty$ then $\mu_\lambda$ is ergodic.
\end{enumerate}
\end{The}

Note that case (a) was proved also by Jung \cite{Jung} for $W = \{0,1\}$. His condition $\sum_x \frac{\lambda_x}{(1+\lambda_x)^2} = \infty$ seems different but is equivalent to the one given here.

\textit{Remarks}

1. In the case that $W= \N$ one might think that it is possible to prove that $\sum_i \lambda_i = \infty$ implies that $\mu_\lambda$ is ergodic without any further conditions like $(\ast \ast)$. We show that this is not possible in section \ref{sec:deterministicprofile}. We give an example of a system where $b$ and $p$ have a specific structure such that there exists a product measures of the given type such that $\sum_i \lambda_i = \infty$, while $\mu_\lambda$ is not ergodic.

This raises the question under which additional assumptions $\sum_i \lambda_i = \infty$ implies ergodicity. The proof of (b) shows some analogy with the proof of theorem 1.8 in Aldous and Pitman \cite{ArtAldousPitman} and the open question we see here is similar to the open question in \cite{ArtAldousPitman}, see theorem 1.8 and example 7.5 in that article.

\bigskip

2. We give an explanation for the symmetric nature of theorem \ref{The:ergodicproduct} (a). We will see that the condition for ergodicity means that the measure concentrates on configurations which have an infinite number of particles i.e. $\sum \eta_i = \infty$, but also such that $\sum_i (N-\eta_i) = \infty$, i.e. infinitely many anti-particles. We give a more intuitive view on this by the following approach. Instead of saying that a particle moves from site $x$ to site $y$ with rate $p(x,y)b(\eta_x,\eta_y)$ one could say that an empty spot, or anti-particle moves from site $y$ to site $x$ with rate $\tilde{p}(y,x) \tilde{b}(N-\eta_y,N-\eta_x)$ where 
\begin{align*}
\tilde{b}(n,k) & = b(N-k,N-n) \\
\tilde{p}(x,y) & = p(y,x) \\
\end{align*}
For more details on this rewrite see section \ref{Sect:antiparticle} below.

\section{Proof of theorem \ref{the:ergoprime} and lemma \ref{Lem:AH}} \label{sec:proofergo}

We start with the proof of lemma \ref{Lem:AH} which states that $\mathcal{A} = \mathcal{H}$. We refer to the point $\phi^{-1}(0)$ as the origin.

\begin{proof}[Proof of lemma \ref{Lem:AH}]
Let $A \in \mathcal{A}$ and fix $n$, we show that $A \in \mathcal{H}_n$. By the defining property of $\mathcal{A}$ we see that $A$ does not depend on the exact configuration of $\eta$ in $S_n$ given its configuration on $S_n^c$ but just on the sum of the values in $S_n$. 
We elaborate on this argument a little for the case that $W = \N$. If one understands the argument for this case then it is clear for the finite case too. Suppose that we have a configuration $\eta \in A$. We see that the configuration $\eta(n)_i =  (\sum_{j \in S_n} \eta_j) \delta_{\phi^{-1}(0)}(i) + \sum_{j \notin S_n} \eta_i \delta_j(i) $ is in $A$ too, because $A \in \mathcal{A}$. So any configuration that is equal to $\eta$ outside $S_n$ and has $\sum_{i \in S_n} \eta_i$ of particles in $S_n$ is in $A$. This means that given the configuration outside $S_n$, $\1_A$ only depends on this $\sum_{i \in S_n} \eta_i$.
Hence $A \in \mathcal{H}_n$, but $n$ was arbitrary, so $A \in \mathcal{H}$.

\bigskip

Let $A \in \mathcal{H}$. Pick a $\eta \in A$ we show that for $x$ and $y$ so that $\eta_x >0$ that $\eta^{x,y} \in A$. Pick a $n$ so that $n > \phi(x),\phi(y)$. We know that $A \in \mathcal{H}_n$ so $A$ does not depend on the exact values in $S_n$ but only on the sum $\sum_{i \in S_n} \eta_i$ which is not changed by moving a particle from $x$ to $y$, therefore $\eta^{x,y} \in A$. This yields $A \in \mathcal{A}$.
\end{proof}

We start with proving theorem \ref{the:ergoprime}, but for this we need some machinery. Fix a measure $\mu \in \mathcal{I}(L)$.

\begin{Prop} \label{prop:extensionL2}
The semigroup $S_t$ on $\overline{D}$ extends to a semigroup $S_t^\mu$ on $\mathcal{L}_2(\mu)$. This in turn defines a unbounded operator $L^\mu$, with domain $\mathcal{D}(L^\mu)$ which is the closure $L$ in $\mathcal{L}_2(\mu)$. $D$ is also a core for $L^\mu$.
\end{Prop}

We denote the norm on $\mathcal{L}_2(\mu)$ by $\vn{\cdot}_\mu$. The proof is rather standard but we give it for sake of completeness in our general setting.

\begin{proof}
By invariance of $\mu$ we obtain that
\begin{equation*}
\vn{S_t f}^2_\mu  \leq  \vn{f}^2_\mu.
\end{equation*}
Hence we see that $S_t$ viewed as a operator on the subset $\overline{D} \subset \mathcal{L}_2(\mu)$ is a contraction. We now prove that $\overline{D}$ is dense in $\mathcal{L}_2(\mu)$. Clearly $D$ contains all local bounded functions, hence its closure in $\mathcal{L}_2(\mu)$ contains all local functions in $\mathcal{L}_2(\mu)$. We prove that all local bounded functions in $\mathcal{L}_2(\mu)$ are dense in $\mathcal{L}_2(\mu)$.

\bigskip

Recall the definitions of $\mathcal{B}_n$. Pick a bounded $f \in \mathcal{L}_2(\mu)$ and define the local functions $f_n = \E[f \; | \; \mathcal \B_n]$. As taking a conditional expectation is a projection in a $\mathcal{L}_2$ space we see that $\vn{f_n}_\mu \leq \vn{f}_\mu$, furthermore the sequence $f_n$ is a martingale with respect to the filtration $(\B_n)_{n \geq 0}$. By martingale convergence $f_n$ converges to $f$ in $\mathcal{L}_2(\mu)$. 

By a truncation argument we see that the bounded functions are dense in $\mathcal{L}_2(\mu)$, so indeed $D$ is dense in $\mathcal{L}_2(\mu)$.

\bigskip

So $S_t$ being a contraction with respect to the Hilbert space norm on $\overline{D} \subset \mathcal{L}_2(\mu)$ defines by a continuous extension a linear operator $S_t^\mu$ on $\mathcal{L}_2(\mu)$. This also defines a generator $L^\mu$ with domain 
\begin{equation*}
\mathcal{D}(L^\mu) := \left\{f \in \mathcal{L}_2(\mu) \; : \; \lim_{t\downarrow 0} \frac{S^\mu_t f - f}{t} \text{ exists in } \mathcal{L}_2(\mu) \right\}.
\end{equation*}
As we clearly have that $\vn{\cdot}_\mu \leq \vn{\cdot}_\infty$, it holds that $L^\mu$ is the closure of $L$ and $D \subset \mathcal{D}(L^\mu)$. As $D$ is a core for $L$ we obtain that $D$ is a core for $L^\mu$ as well. 

This last statement is obtained by using proposition 3.1 from Ethier and Kurtz \cite{EthierKurtz}. This proposition shows that $\mathcal{R}(\lambda - L)$ is dense in $\overline{D}$ for some $\lambda >0$. We know that $D$ is dense in $\mathcal{L}_2(\mu)$, hence $\mathcal{R}(\lambda - L^\mu)$ is dense in $\mathcal{L}_2(\mu)$. The same proposition yields that $D$ is a core for $L^\mu$.

\end{proof}

We now give a technical result which helps us to analyse the structure of the set $\mathcal{I}$.
Define in the spirit of lemma IV.4.3 of Liggett \cite{BookLiggett1} and Sethuraman \cite{SSEx} the following two quadratic forms, for $f$ for which they are finite:
\begin{align*}
Q(f) & = - \E_\mu[f L^\mu f] \\
R(f) & = \frac{1}{2} \sum_T \E_\mu \left[ \int c_T(\eta,\dd \zeta) (\nabla_{T,\zeta}f(\eta))^2\right]
\end{align*}

Liggett defines bilinear forms instead of quadratic ones, we will not do that here because the following result is only true for quadratic forms. Below we will show that equality for bilinear forms is possible only in the case that the underlying measure $\mu$ is reversible with respect to the dynamics. 

\begin{Prop} \label{prop:Dirichletform}
For $f \in \mathcal{D}(L^\mu)$:
\begin{equation*}
Q(f) = R(f) < \infty
\end{equation*}
\end{Prop}

\textit{Remark} This proposition is an improvement over lemma 2.4 of Sethuraman \cite{SSEx} since the latter only holds for product measures.

\begin{proof}
The proof is analogous to that of lemma IV.4.3 in Liggett \cite{BookLiggett1}. We will not repeat the proof here, the key step that is different is to note that for $f \in D$ it holds that
\begin{equation*}
\E_\mu[f L^\mu f] = \E_\mu[f L^\mu f] - \frac{1}{2} \E_\mu[L^\mu f^2].
\end{equation*}
After that simply work out the right hand side and plug in the arguments from \cite{BookLiggett1}.
\end{proof}

The same techniques can be used to prove that for $f,g \in \mathcal{D}(L^\mu)$
\begin{equation*}
-\E_\mu[fL^\mu g + gL^\mu f] = \sum_T \E_\mu \left[ \int c_T(\eta,\dd \zeta) (\nabla_{T,\zeta}f(\eta))(\nabla_{T,\zeta}g(\eta))\right]
\end{equation*}
by using that $\E_\mu[fL^\mu g + gL^\mu f] = \E_\mu[fL^\mu g + gL^\mu f - L^\mu(fg)]$. This shows that we have equality for bilinear forms only when $\mu$ is reversible with respect to the dynamics.

\bigskip

For the proof of theorem \ref{the:ergoprime} we introduce approximating Markov processes. Recall the definition of $S_n$.
Define for $f \in D$
\begin{equation*}
L^{(n)}f(\eta) = \sum_{T \subset S_n} \int c_T(\eta,\dd \zeta) (f(\theta_{T,\zeta}(\eta)) - f(\eta)).
\end{equation*}
Because $S_n$ is a finite set $L^{(n)}$ is a bounded operator which therefore generates a Markov Jump process with semigroup $S_t(n)$. This semigroup also extends to $S_t^\mu(n)$ on $\mathcal{L}_2(\mu)$.

\begin{proof}[Proof of theorem \ref{the:ergoprime}]
Suppose that $\mu$ is ergodic. Pick a set $A \in \mathcal{G}_L$, we need to show that $\mu(A) \in \{0,1\}$ or equivalently that the function $\1_A$ is constant $\mu$ almost surely. Intuitively one would like to say that $L^\mu \1_A = 0$, because clearly for every $\eta$, finite $T \subset S$ and $\zeta \in W^T$  it holds that $\nabla_{T,\zeta} \1_A(\eta) = 0$, hence $S^\mu_t \1_A = \1_A$ for all $t$, hence by ergodicity $\1_A$ is constant $\mu$ almost surely.

\bigskip

This reasoning is not rigorous as we do not know whether $\1_A \in \mathcal{D}(L^\mu)$. However by corollary I.3.14 in Liggett \cite{BookLiggett1} we obtain that $S_t^\mu(n) f \rightarrow S_t^\mu f$ for all $f \in \mathcal{L}_2(\mu)$, uniformly for $t$ in compact intervals.

\bigskip

The set $A \in \mathcal{G}_L$ is invariant under finitely many transformations of the form $\eta$ to $\theta_{T,\zeta}(\eta)$ for $T$ and $\zeta$ so that $c_T(\eta,\zeta) > 0$. Denote the Markov process generated by $L^{(n)}$ by $\eta_n(t)$. Under the law of this Markov process the set on which there are only a finite number of allowed transitions by time $t$ has probability $1$. This means that for any starting configuration $\eta$, $t \geq 0$ and $n \in \N$ it holds that $S_t^\mu(n) \1_A(\eta) = \E_\eta[1_A(\eta_n(t))] = 1_A(\eta)$. Hence for every $t$ and $n$ it holds that $S_t^\mu(n) \1_A = \1_A$ in the space $\mathcal{L}_2(\mu)$. Furthermore $S_t^\mu(n) \1_A \rightarrow S_t^\mu \1_A$, hence $S_t^\mu \1_A = \1_A$ in $\mathcal{L}_2(\mu)$ for every $t$.

\bigskip

Now we can use the ergodicity of $\mu$ with respect to the Markov process to obtain that $1_A$ is $\mu$ constant almost surely, which implies that $\mu(A) \in \{0,1\}$. Since $A \in \mathcal{G}_L$ was arbitrary we see that $\mathcal{G}_L$ is trivial under $\mu$.

\bigskip

For the second implication assume that $\mathcal{G}_L$ is trivial under $\mu$. Fix $A \in \mathcal{B}$ and assume that $S_t^\mu \1_A = \1_A \; \mu $ a.s. for all $t\geq 0$. We will show that there is a set $A_\infty \in \mathcal{G}_L$ such that $\mu(A) = \mu(A_\infty)$. First note that $\1_A \in \mathcal{D}(L^\mu)$ and $L^\mu \1_A = 0$, hence by proposition \ref{prop:Dirichletform} we see that $R(f) = 0$. This in turn implies that the set $B_0$ defined by
\begin{equation*}
\left\{\eta \; : \; \exists \; T \subset S \text{ finite}, \zeta \in W^T, \text{ such that } c_T(\eta,\zeta) > 0, \1_A(\eta) \neq \1_A(\theta_{T,\zeta}(\eta)) \right\} 
\end{equation*}
has $\mu$ measure zero. Let $A_0 = A$. Define $A_1 = A_0 \setminus B_0$ and note that $\1_{A_0} = \1_{A_1} \; \mu$ a.s. because $\mu(B_0) = 0$. This means that 
\begin{equation*}
S_t^\mu \1_{A_1} = S_t^\mu \1_{A_0} = \1_{A_0} = \1_{A_1} \quad \mu \text{ a.s.}
\end{equation*}
This yields that the set $B_1$ given by 
\begin{equation*}
\left\{\eta \; : \; \exists \; T \subset S \text{ finite }, \zeta \in W^T, \text{ such that } c_T(\eta,\zeta) > 0, \1_{A_1}(\eta) \neq \1_{A_1}(\theta_{T,\zeta}(\eta)) \right\} 
\end{equation*}
has measure $0$. Define $A_2 = A_1 \setminus B_1$. We can repeat this step and construct $A_3$, $A_4$, $\dots$. Note that $A_{n+1} \subset A_n$ and $\mu(A_{n+1}) = \mu(A_n)$ for all $n$. Define $A_\infty = \bigcap_n A_n$ and note that $\mu(A_\infty) = \mu(A)$.

\bigskip

We show that $A_\infty \in \mathcal{G}_L$. Suppose $\eta \in A_\infty$, $T \subset S$ finite and $\zeta \in W^S$ such that $c_T(\eta,\zeta) > 0$, we must prove that $\theta_{T,\zeta}(\eta) \in A_\infty$. This is not to difficult, suppose that $\theta_{T,\zeta}(\eta) \notin A_\infty$, then there is a $N >0$ so that for all $n \geq N \;$ $\theta_{T,\zeta}(\eta) \notin A_n$, but then for all $n > N \;$ $\eta \notin A_n$, so that it follows that $\eta \notin A$ which is a contradiction.

\bigskip

This means that $\mu(A) = \mu(A_\infty) \in \{0,1\}$ because of triviality of $\mathcal{G}_L$ under $\mu$.
\end{proof}

\section{Proof of proposition \ref{prop:invariantmeasures}} \label{sect:product}

First we give a consequence of the definition of the product measures in $\mathcal{P}_\otimes(b)$.
Let $A$ be the set defined by
\begin{equation*}
A = \begin{cases}
\{\eta_y >0\}\cap\{\eta_x< N \} & \text{ if }  W = \{0,\dots,N\} \\
\{\eta_y >0\} & \text{ if }  W = \N.
\end{cases}
\end{equation*}
On the set $A$ define $\mu_\lambda^{y,x}$ to be the measure obtained from $\mu_\lambda$ by the transformation $\eta \mapsto \eta^{y,x}$, i.e. $\1_A \mu_\lambda^{y,x}(\dd \eta) = \1_A \mu_\lambda(\dd \eta^{y,x})$. 

\begin{Lem} \label{lem:RNderivative}
For $\mu_\lambda$ the Radon-Nikodym derivative corresponding to the change of variables $\eta^{y,x}$ to $\eta$ is:
\begin{equation*}
\1_A \frac{\dd \mu_\lambda^{y,x}}{\dd \mu_\lambda}(\eta) = \frac{b(\eta_y,\eta_x)}{b(\eta_x+1,\eta_y -1)} \frac{\lambda_x}{\lambda_y}
\end{equation*}
\end{Lem}

\begin{proof}
The transformation $\eta \mapsto \eta^{y,x}$ only affects two coordinates, hence
\begin{align*}
\1_A \frac{\dd \mu_\lambda^{y,x}}{\dd \mu_\lambda}(\eta) & = \1_A \frac{a_{\eta_x+1}}{a_{\eta_x}} \frac{a_{\eta_y-1}}{a_{\eta_y}}  \frac{\lambda_x}{\lambda_y}\\
& = \frac{b(1,\eta_x)}{b(\eta_x+1,0)} \frac{b(\eta_y,0)}{b(1,\eta_y-1)} \frac{\lambda_x}{\lambda_y} \\
& = \frac{b(\eta_y,\eta_x)}{b(\eta_x+1,\eta_y -1)} \frac{\lambda_x}{\lambda_y}.
\end{align*}
In the last line we use assumption \ref{Ass:1}.
\end{proof}

We start the proof of proposition \ref{prop:invariantmeasures} with:
\begin{Lem} \label{lem:revp} 
Let $\lambda : S \rightarrow \R^+$ be a solution of $\sum_x \lambda_x p(x,y) = \sum_x \lambda_y p(y,x)$ and suppose that if $\lambda_x p(x,y) \neq \lambda_y p(y,x)$, then $\lambda_x = \lambda_y$.
Define the relation $x \thicksim y$ if $\lambda_x p(x,y) = \lambda_y p(y,x)$. Then it holds that
\begin{equation*}
\sum_{x \nsim y} p(x,y) = \sum_{x \nsim y} p(y,x).
\end{equation*}
\end{Lem}

\begin{proof}
This is a short calculation.
\begin{align*}
\sum_x \lambda_y p(y,x)& = \sum_x \lambda_x p(x,y) \\
& = \sum_{x \thicksim y} \lambda_x p(x,y) + \sum_{x \nsim y} \lambda_x p(x,y) \\
& = \sum_{x \thicksim y} \lambda_y p(y,x) + \sum_{x \nsim y} \lambda_y p(x,y)
\end{align*}
\end{proof}

\begin{proof}[Proof of proposition \ref{prop:invariantmeasures}]
Fix $L^{b,p}$ and pick a measure $\mu_\lambda \in \mathcal{P}_\otimes(b)$.

Let $f \in D$. We are allowed to rearrange the terms in the next calculation, because $f$ is a local function and $p$ is finite range.
\begin{equation} \label{eqn:omzetting}
\begin{aligned}
\int Lf \dd \mu_\lambda & = \int \sum_{x,y} p(x,y) b(\eta_x,\eta_y) (f(\eta^{x,y}) - f(\eta)) \mu_\lambda(\dd \eta) \\
& = \int \sum_{x,y} p(x,y) b(\eta_x^{y,x},\eta_y^{y,x}) f(\eta) \mu_\lambda^{y,x}(\dd \eta) \\
& \quad \quad - \int \sum_{x,y} p(x,y) b(\eta_x,\eta_y) f(\eta) \mu_\lambda(\dd \eta) \\
& = \int \sum_{x,y} p(x,y) \frac{\lambda_x}{\lambda_y} b(\eta_y,\eta_x) f(\eta) \mu_\lambda(\dd \eta) \\
& \quad \quad - \int \sum_{x,y} p(x,y) b(\eta_x,\eta_y) f(\eta) \mu_\lambda(\dd \eta) \\
& = \int f(\eta) \sum_{x,y} b(\eta_x, \eta_y) \left[ p(y,x) \frac{\lambda_y}{\lambda_x} - p(x,y) \right] \mu_\lambda(\dd \eta)
\end{aligned}
\end{equation}

We arrive at the fourth line by using lemma \ref{lem:RNderivative} on the first term. We obtain the last expression by changing the roles of $x$ and $y$ in the first term. 

\bigskip

Clearly if for all $x,y$ that $\lambda_x p(x,y) = \lambda_y p(y,x)$ the integral is $0$. This is case (c). Now suppose that this is not the case and we have a pair $i,j$ such that $\lambda_i p(i,j) \neq \lambda_j p(j,i)$. Then the above argument does not work. Note that if we can prove that that for all $\eta_x, \eta_y$ it holds that $\sum_{x,y} b(\eta_x, \eta_y) [ p(y,x) \frac{\lambda_y}{\lambda_x} - p(x,y) ] = 0$, then we are done. We start with 1.15(b), suppose that if $x \nsim y$ then $\lambda_x = \lambda_y$.
\begin{align*}
& \sum_{x,y} b(\eta_x, \eta_y) \left[ p(y,x) \frac{\lambda_y}{\lambda_x} - p(x,y) \right] \\
& = \sum_x \sum_{y \nsim x} b(\eta_x, \eta_y) [p(y,x) - p(x,y)] \\
& = \sum_x \sum_{y \nsim x} b(\eta_x, \eta_y) p(y,x) - \sum_x \sum_{y \nsim x} b(\eta_x, \eta_y) p(x,y) \\
& = \sum_x \sum_{y \nsim x} b(\eta_x, \eta_y) p(y,x) - \sum_y \sum_{x \nsim y} b(\eta_y, \eta_x) p(y,x) \\
& = \sum_x \sum_{y \nsim x} b(\eta_x, \eta_y) p(y,x) - \sum_x \sum_{y \nsim x} b(\eta_y, \eta_x) p(y,x) \\
& = \sum_x \sum_{y \nsim x} p(y,x) [b(\eta_x, \eta_y) - b(\eta_y,\eta_x)] \\
& = \sum_x \sum_{y \nsim x} p(y,x) [b(\eta_x, 0) - b(\eta_y,0)] \\
& = \sum_x \sum_{y \nsim x} b(\eta_x,0) [p(y,x) - p(x,y)] \\
& = \sum_x b(\eta_x,0) \sum_{y \nsim x}  [p(y,x) - p(x,y)] \\
& = 0
\end{align*}
In line five we use that $\thicksim$ is a symmetric relation. In line seven we use the second item in the assumptions. and in line eight we switch back the way we switched forward in lines two to five. In the last line we use lemma \ref{lem:revp}.

For the proof of item (a) note that the method above does not work in this case as we cannot use reversibility or the relation $\thicksim$. However $b(n,k)$ reduces to $b(n,0)$. We check again that $\sum_{x,y} b(\eta_x, \eta_y) [ p(y,x) \frac{\lambda_y}{\lambda_x} - p(x,y) ] = 0$.
\begin{align*}
& \sum_{x,y} b(\eta_x, \eta_y) \left[ p(y,x) \frac{\lambda_y}{\lambda_x} - p(x,y) \right] \\
& = \sum_x b(\eta_x,0) \sum_ y \left[ p(y,x) \frac{\lambda_y}{\lambda_x} - p(x,y) \right] \\
& = 0
\end{align*}
The last equality is due to the primary assumption on $\lambda$: $\sum_x \lambda_x p(x,y) = \lambda_y \sum_x p(y,x)$.

\bigskip

We now prove that an invariant measure in the set $\mathcal{P}_\otimes(b)$ must be of one of the three given types. Pick a point $z \in S$ and a finite set $B(z)$ containing $z$ such that if $x \notin B(z)$, then $p(z,x) = p(x,z)= 0$, i.e. $B(z)$ contains all points that can be reached by $p$ from $z$. Let $F(z) = \{\eta \; : \; \text{ if } x \in B(z) \setminus \{z\} \text{ then } \eta_x = 0 \}$, furthermore let $F^*(z) = \{\eta \; : \; \text{ if } x \in B(z) \text{ then } \eta_x = 0 \}$. Note that $\1_{F(z)}$ and $\1_{F^*(z)}$ are local bounded functions, hence in $D$. Now fix some generator $L = L^{b,p}$ for which we know $p$, but do not know the specific form of $b$.

\bigskip

Suppose that we have a product measure $\mu_\lambda \in \mathcal{P}_\otimes(b)$, for a nonzero $\lambda$ and suppose that $\mu_\lambda$ is invariant for the process generated by $L^{b,p}$. This yields $\int L\1_{F(z)} \dd \mu_\lambda = 0$ and $\int  L\1_{F^*(z)} \dd \mu_\lambda = 0$ by equation (\ref{eqn:I}). We now look at these integrals. First note that by definition of our indicator function, we only have to look at couples where $x$ or $y$ is in $B(z)$. The second equality is due to a calculation similar to the one in \eqref{eqn:omzetting}.
\begin{equation} \label{eqn:lambdacondition1}
\begin{aligned} 
0 & = \int L\1_{F(z)} \dd \mu_\lambda \\
& = \sum_{x \text{ or } y \in B(z)}  p(x,y) \left(\1_{F(z)}(\eta^{x,y}) - \1_{F(z)}(\eta) \right) \mu_\lambda(\dd \eta) \\
& = \int \1_{F(z)}(\eta) \sum_{y \in B(z)} b(\eta_z,0) \left[p(y,z) \frac{\lambda_y}{\lambda_z} - p(z,y) \right]  \mu_\lambda(\dd \eta) \\
& \quad + \int \1_{F(z)}(\eta) \sum_{x \notin B(z)} \sum_{y \in B(z)} b(\eta_x,0) \left[p(y,x) \frac{\lambda_y}{\lambda_x} - p(x,y) \right] \mu_\lambda(\dd \eta)
\end{aligned}
\end{equation} 
When using the same methods on the function $\1_{F^*(z)} \mu_\lambda(\eta_z = 0)^{-1}$ we obtain 
\begin{equation*} 
\begin{aligned}
0 & = \mu_\lambda(\eta_z = 0)^{-1} \int L\1_{F^*(z)} \dd \mu_\lambda \\
& = \mu_\lambda(\eta_z = 0)^{-1} \int \1_{F^*(z)}(\eta) \sum_{x \notin B(z)} \sum_{y \in B(z)} b(\eta_x,0) \left[p(y,x) \frac{\lambda_y}{\lambda_x} - p(x,y) \right] \mu_\lambda(\dd \eta)
\end{aligned}
\end{equation*} 
We know that $\mu_\lambda$ is a product measure and in the last line the only term involving the integral over $\eta_z$ is the function $\1_{F^*(z)}$, but clearly $\1_{F^*(z)} = \1_{F(z)} \1_{\{\eta_z = 0\}}$. Hence we can first integrate over $\eta_z$ such that the normalising term disappears and then add the integral over $\eta_z$, because it integrates to $1$:
\begin{equation} \label{eqn:lambdacondition2}
\begin{aligned}
0 & = \mu_\lambda(\eta_z = 0)^{-1} \int \1_{F^*(z)}(\eta) \sum_{x \notin B(z)} \sum_{y \in B(z)} b(\eta_x,0) \left[p(y,x) \frac{\lambda_y}{\lambda_x} - p(x,y) \right] \mu_\lambda(\dd \eta) \\
& = \int \1_{F(z)}(\eta) \sum_{x \notin B(z)} \sum_{y \in B(z)} b(\eta_x,0) \left[p(y,x) \frac{\lambda_y}{\lambda_x} - p(x,y) \right] \mu_\lambda(\dd \eta)
\end{aligned}
\end{equation} 
Combining \eqref{eqn:lambdacondition1} and \eqref{eqn:lambdacondition2} we obtain that
\begin{equation} \label{eqn:uitkomstpfinite1}
0 = \int \1_{F(z)}(\eta) \sum_{y \in B(z)} b(\eta_z,0) \left[p(y,z) \frac{\lambda_y}{\lambda_z} - p(z,y) \right] \mu_\lambda(\dd \eta)
\end{equation}
After integrating over $\eta_z$ we obtain that $ \sum_y \left[p(y,z) \frac{\lambda_y}{\lambda_z} - p(z,y) \right] = 0$, hence $\lambda$ solves $\sum_y \lambda_y p(y,z) = \lambda_z \sum_y p(z,y)$.

\bigskip

For the subdivision into items (1), (2) and (3) we adapt the above argument by looking at two sites. Pick two distinct sites $z$ and $w$ such that $p(z,w)$ or $p(w,z)$ is nonzero. Fix a finite set $B(z,w) \subset S$ containing $z$ and $w$ such that if $y \notin B(z,w)$ then $p(z,y) = p(y,z) = p(w,y) = p(w,y) = 0$. Let $F(z,w) = \{\eta \; : \; \text{ if } x \in B(z,w) \setminus \{z, w\} \text{ then } \eta_x = 0, \eta_z = n, \eta_w = k \}$ and let $F^*(z,w) = \{\eta \; : \; \text{ if } x \in B(z,w) \text{ then } \eta_x = 0 \}$. We do a similar calculation.

\begin{align*}
0 & = \int L\1_{F(z,w)}  \dd \mu_\lambda \\
& = \int \1_{F(z,w)}(\eta) b(\eta_z, \eta_w) \left[ p(w,z) \frac{\lambda_w}{\lambda_z} - p(z,w) \right]\mu_\lambda(\dd \eta) \\
& \quad + \int \1_{F(z,w)}(\eta) \sum_{ y \in B(z,w)\setminus \{w\}} b(\eta_z, \eta_y) \left[ p(y,z) \frac{\lambda_y}{\lambda_z} - p(z,y) \right]\mu_\lambda(\dd \eta) \\
& \quad + \int \1_{F(z,w)}(\eta) b(\eta_w, \eta_z) \left[ p(z,w) \frac{\lambda_z}{\lambda_w} - p(w,z) \right]\mu_\lambda(\dd \eta) \\
& \quad + \int \1_{F(z,w)}(\eta)  \sum_{ y \in B(z,w)\setminus \{z\}} b(\eta_w, \eta_y) \left[ p(y,w) \frac{\lambda_y}{\lambda_w} - p(w,y) \right]\mu_\lambda(\dd \eta) \\
& \quad + \int \1_{F(z,w)}(\eta) \sum_{x \notin B(z,w)} \sum_{y \in B(z,w)} b(\eta_w, \eta_y) \left[ p(y,w) \frac{\lambda_y}{\lambda_x} - p(x,y) \right]\mu_\lambda(\dd \eta) \\
\end{align*}
We clarify the last expression. We have split the sum over $x$ and $y$ into a number of parts, in the first two lines $x = z$, in the second two lines $x = w$ and in the last line we sum over $x \notin B(z,w)$. The sum over $x \in B(z,w) \setminus \{z,w \}$ does not play a role because on the set $F(z,w)$ we integrate over $b(0,\cdot) = 0$.
The term in the last line is $0$ because it is equal to
\begin{equation*}
\frac{\mu_\lambda(\eta_z = n)}{\mu_\lambda(\eta_z = 0)} \frac{\mu_\lambda(\eta_w = k)}{\mu_\lambda(\eta_w = 0)} \int L \1_{F^*(z,w)} \dd \mu_\lambda
\end{equation*}
just like in the argument where we singled out only one point in $S$. We obtain that
\begin{equation*}
\begin{aligned}
0 & =  \sum_{y \neq w} b(n,0)\left[ p(y,z) \frac{\lambda_y}{\lambda_z} - p(z,y)\right] \\
& \quad +  \sum_{y \neq z} b(k,0)\left[ p(y,w) \frac{\lambda_y}{\lambda_w} - p(w,y)\right] \\
& \quad +  b(n,k) \left[ p(w,z) \frac{\lambda_w}{\lambda_z} - p(z,w)\right] \\
& \quad +  b(k,n) \left[ p(z,w) \frac{\lambda_z}{\lambda_w} - p(w,z)\right]
\end{aligned}
\end{equation*}

Now we add the terms missing from the first two sums and subtract them from the last two sums:
\begin{equation*}
\begin{aligned}
0 & = b(n,0) \sum_{y}\left[ p(y,z) \frac{\lambda_y}{\lambda_z} - p(z,y)\right] \\
& \quad + b(k,0) \sum_{y} \left[ p(y,w) \frac{\lambda_y}{\lambda_w} - p(w,y)\right] \\
& \quad +  \left[b(n,k) - b(n,0)\right] \left[ p(w,z) \frac{\lambda_w}{\lambda_z} - p(z,w)\right] \\
& \quad +  \left[b(k,n) - b(k,0)\right] \left[ p(z,w) \frac{\lambda_z}{\lambda_w} - p(w,z)\right]
\end{aligned}
\end{equation*}

Note that the first 2 lines are zero because $\lambda$ satisfies 
\begin{equation*}
\sum_x \lambda_x p(x,y) = \sum_x \lambda_y p(y,x). 
\end{equation*}
Therefore:
\begin{equation} \label{eqn:taggedsufficient}
\begin{aligned}
0 & =   \left[b(n,k)-b(n,0) \right]\left[ p(w,z) \frac{\lambda_w}{\lambda_z} - p(z,w)\right] \\
& \quad +  \left[b(k,n)-b(k,0) \right]\left[ p(z,w) \frac{\lambda_z}{\lambda_w} - p(w,z)\right]
\end{aligned}
\end{equation}
From equation (\ref{eqn:taggedsufficient}) we can now derive necessary conditions for $\lambda$ to yield an invariant measure for a certain $b$.

\bigskip

Suppose that $b$ does not depend on the second variable, $b(n,k) = b(n,0)$, then we see that this yields $0$ in equation (\ref{eqn:taggedsufficient}). This is option (a) in the proposition. 
Now suppose that we have a $k$ so that $b(n,k) \neq b(n,0)$. We can rewrite equation (\ref{eqn:taggedsufficient}) to the following equality:
\begin{align*}
& \left[b(n,k) - b(n,0) \right] \left[p(w,z) \frac{\lambda_w}{\lambda_z} - p(z,w) \right] \\
& \quad = \frac{\lambda_z}{\lambda_w}\left[b(k,n) - b(k,0) \right] \left[p(w,z) \frac{\lambda_w}{\lambda_z} - p(z,w) \right]
\end{align*}
This gives by the assumption that $p(z,w) + p(w,z) >0$ either 
\begin{equation} \label{eqn:taggedsuff1}
\left[b(n,k) - b(n,0) \right] = \frac{\lambda_z}{\lambda_w}\left[b(k,n) - b(k,0) \right]
\end{equation}
or
\begin{equation} \label{eqn:taggedsuff2}
p(w,z) \frac{\lambda_w}{\lambda_z} - p(z,w) = 0
\end{equation}
Now we can simply take $\lambda$ to be reversible, which is exactly equation (\ref{eqn:taggedsuff2}). This is option (c) in the proposition. Suppose that we have two sites $z$ and $w$ so that $\lambda$ is not reversible: $p(w,z) \frac{\lambda_z}{\lambda_w} - p(z,w) \neq 0$, then equation (\ref{eqn:taggedsuff1}) must hold.

\bigskip
Now suppose that $\frac{\lambda_z}{\lambda_w} = c \neq 1$, then $b(n,k) - b(n,0) = c( b(k,n) - b(k,0))$. Note that because $b(n,k) \neq b(n,0)$ clearly $b(k,n) \neq b(k,0)$, so we could have made the same argument with $n$ and $k$ the other way around to obtain that $b(n,k) - b(n,0) = c^{-1}( b(k,n) - b(k,0))$, a contradiction.
Hence the assumption that $p(w,z) \frac{\lambda_w}{\lambda_z} - p(z,w) \neq 0$ leads to the fact that $\lambda_z = \lambda_w$. And this in turn leads by equation (\ref{eqn:taggedsuff1}) to $b(n,k) - b(k,n) = b(n,0) - b(k,0)$. This is option (b) in the proposition.

\end{proof}

\section{Proof of theorem \ref{The:ergodicproduct} and a counterexample} \label{sect:proof}

We prove the following two theorems which imply theorem \ref{The:ergodicproduct} by using corollary \ref{Cor:ergodicity}. Let $\mu_\lambda \in \mathcal{I}(L^{b,p}) \cap \mathcal{P}_\otimes(b)$, which we will denote by $\PR$ in this section.

\begin{The} \label{The:equivA}
In the case that $W = \{0, \dots, N\}$ the following are equivalent.
\begin{enumerate}
\item $\mathcal{A}$ is trivial.
\item $ \sum_{i : \lambda_i < 1} \PR[\eta_i > 0] + \sum_{i: \lambda_i \geq 1} \PR[\eta_i < N] = \infty $
\item $ \sum_{i : \lambda_i < 1} \lambda_i + \sum_{i: \lambda_i \geq 1} \frac{1}{\lambda_i} = \infty$
\end{enumerate}
\end{The}

In the case that $W = \N$ there is no equivalent of $N$, so we need to adjust (2) and (3).

\begin{The} \label{The:equivA2}
Under the assumption $(\ast \ast)$ that $\lambda^*<\infty$ or that $\lambda^* = \infty$ and there exists a finite set $\mathfrak{D} = \{d_1, \dots, d_p\}$ so that $\gcd(\mathfrak{D}) = 1$ such that
\begin{equation}
\mathfrak{D} \subset \left\{d \geq 1 \; : \; \sup_{k} \frac{a_k^2}{a_{k-d}a_{k+d}} = \sup_k \prod_{i=0}^{d-1}\frac{b(k+i+1,k-i-1)}{b(k-i,k+i)}< \infty \right\}
\end{equation}
then we have the following equivalence.
\begin{enumerate}
\item $\mathcal{A}$ is trivial.
\item $\sum_i \PR[\eta_i > 0] = \infty$.
\item $\sum_i \lambda_i = \infty$.
\end{enumerate}
\end{The}

We leave the proof of the equivalence of (2) and (3) to the reader, as this is straightforward by the definition in (\ref{eqn:defa}). The proof that (1) implies (2) is a consequence of Borel-Cantelli. Note that we do not need to assume $(\ast \ast)$ for these arguments.

\bigskip

The proof of (3) to (1) in both theorems uses a coupling argument. 

From lemma \ref{Lem:AH} we know that $\mathcal{A} = \mathcal{H}$, so we look at $\mathcal{H}$ instead. As $S$ is a countable set, we assume for the moment that it is equal to $\N$ to simplify the notation. Now define the partial sums $Z_n = \sum_{i =0}^n \eta_i$, then $\mathcal{H}$ is the tail $\sigma$-algebra of the $Z_n$.

\bigskip

Define transition matrices $p_n$ for every $n \in \N$ by $p_n(x,y) = \mu_{\lambda_n}(y-x)$. We see that the chain $(Z_n)_{n \geq 0}$ is a time inhomogeneous Markov chain with transition matrices $(p_n)_{n \geq 0}$. Let $p((t_0, k_0),(t,k))$ give the probability that the chain that is started at time $t_0$ in $k_0$ is in $k$ at time $t$, in other words $p((t_0, k_0),(t,k)) = (\prod_{i = t_0+1}^t p_i)(k_0,k)$.

\bigskip

 The next theorem is theorem 4 of Iosifescu \cite{IosifescuTail}, but can also be derived from theorem 4.1 in Thorisson \cite{Thorisson} or from theorem 20.10 in Kallenberg \cite{Kallenberg}.

\begin{The} \label{The:basiccoupling}
$\mathcal{H}$ is trivial if and only if
\begin{equation*}
\lim_{n \rightarrow \infty} \sum_{j} \left| p((n_0,s_0),(n,j)) - p((0,0),(n,j)) \right| = 0
\end{equation*}
for all $n_0 \in \N$ and $s_0 \in \N$.
\end{The}
Note that if we can prove for all $n_0, s_0 \in \N$ that
\begin{equation*}
\lim_{n \rightarrow \infty} \sum_{s} \PR[Z_{n_0} = s] \sum_{j} \left| p((n_0,s_0),(n,j)) - p((n_0,s),(n,j)) \right| = 0
\end{equation*}
then we satisfy the condition of the theorem. Furthermore note that for a fixed $s$ the sum over $j$ is two times the total variation distance between two chains starting at time $n_0$ in the point $s_0$ and in the point $s$. So if we can show that the total variation distance converges to $0$ for any point $s$ then by the dominated convergence theorem also the sum over $s$ converges to zero.

\begin{Cor} 
If we have a successful coupling of two chains starting at time $n_0$ at $s$ and at $s_0$, for every $n_0, s, s_0$, then $\mathcal{H}$ is trivial.
\end{Cor}
\begin{proof}
Let $T$ be the coupling time. A successful coupling means that $\PR[T < \infty] = 1$. Let $Y$ be the chain started at time $n_0$ at $s_0$ and let $\hat{Y}$ be the chain that starts at time $n_0$ at site $s$.
\begin{equation*}
2 \vn{Y_n - \hat{Y}_n}_{TV} = \sum_{j} \left| p((n_0,s_0),(n,j)) - p((n_0,s),(n,j)) \right|
\end{equation*}
By the coupling event inequality which can be found as equation (2.10) in Lindvall \cite{LindvallCoupling} we know that
\begin{equation*}
\vn{Y_n - \hat{Y}_n}_{TV} \leq 2 \PR[T > n - n_0]
\end{equation*}
By letting $n$ go to infinity we obtain the result.
\end{proof}

We need to construct a successful coupling of two chains starting at time $n_0$ at positions $s_0$ and $s$. We do this by the so called Mineka coupling, which is described in Lindvall \cite{LindvallCoupling}. Let $\alpha^i_k = 1/2 \; (\mu_{\lambda_i}(k) \wedge \mu_{\lambda_i}(k+1))$. Let $R_i$ and $\hat{R}_i$ be the step sizes of the coupled random walks, their probabilities are given by:
\begin{align*}
\PR[(R_i,\hat{R}_i) = (k-1,k)] & = \alpha^i_{k-1} \\
\PR[(R_i,\hat{R}_i) = (k,k-1)] & = \alpha^i_{k-1} \\
\PR[(R_i,\hat{R}_i) = (k,k)] & = \mu_{\lambda_i}(k) - \alpha^i_k - \alpha^i_{k-1}
\end{align*}  
Now let $Y_n = s_0 + \sum_{i=n_0 + 1}^n R_i$ and $\hat{Y}_n = s + \sum_{i= n_0 + 1}^n \hat{R}_i$ for $n  \geq n_0$ as long as $Y_n \neq \hat{Y}_n$. From the first moment that $Y_n = \hat{Y}_n$ we let $Y$ and $\hat{Y}$ make the same steps of which the distribution is given by the measure $\mu_\lambda$. It is easy to check that the distributions of $Y$ and $\hat{Y}$ are correct. If we look at the difference of these two chains: $V_n = Y_n - \hat{Y}_n$ we see that $V_n$ starts in $Y_{n_0} - \hat{Y}_{n_0} = s_0 - s$ and $(V_n)_n$ makes steps of size one and of size zero only. Furthermore the expectation of steps is zero. 
As long as $V_n$ has not hit zero yet, then it will make a random walk given by the transition probabilities for $R$ and $\hat{R}$. It is a well known fact that if such a random walk makes an infinite number of non-zero steps, then it is recurrent, see e.g. \cite{Spitzer}. So if we can prove that the chain induced by $R$ and $\hat{R}$ makes an infinite number of steps, then $V$ will hit zero eventually and thus the coupling is successful. 

\begin{Lem} \label{lem:demandcoupling}
The coupling is successful if:
\begin{equation*}
\sum_i \sum_k \PR[\eta_i = k] \wedge \PR[\eta_i = k+1] = \infty
\end{equation*}
\end{Lem}

\begin{proof}
The random walk given by $R$ and $\hat{R}$ makes an infinite number of non-zero steps if $\sum_k \sum_i 2 \alpha_k^i = \infty$ by Borel-Cantelli. But this is exactly the claim.
\end{proof}

We give a small summary in the form of the next corollary.

\begin{Cor} \label{Cor:basictrivialitydemand}
$\mathcal{A}$ is trivial under the product measure $\PR$ if
\begin{equation} \label{eqn:basictrivialitydemand}
\sum_i \sum_k \PR[\eta_i = k] \wedge \PR[\eta_i = k+1] = \infty 
\end{equation}
\end{Cor}

The result can be generalised, for this we need Bezout's identity. The proof of this lemma is elementary, see e.g. \cite{JonesJonesNumberTheory,UspenskyHeasletElNumberTheory}.

\begin{Lem} \label{lem:Bezout}
For a finite set $\mathfrak{D} = \{d_1, \dots, d_p\}$ of positive integers there are $x_1 , \dots, x_p \in \Z$ so that:
\begin{equation*}
\gcd(\mathfrak{D}) = x_1 d_1 + \dots + x_p d_p
\end{equation*}
Furthermore $\gcd(\mathfrak{D})$ is the smallest positive integer for which this is possible.
\end{Lem}

We use it to prove the following theorem.

\begin{The} \label{The:trivialitydemand}
Suppose that there exists a finite set $\mathfrak{D} = \{d_1, \dots, d_p\}$ for some $p$ such that
\begin{equation*}
\mathfrak{D} \subset \left\{ d > 0 \; : \; \sum_i \sum_k \PR[\eta_i = k+d] \wedge \PR[\eta_i = k] = \infty \right\}
\end{equation*}
and $\gcd(\mathfrak{D}) = 1$, then $\mathcal{A}$ is trivial under the product measure $\PR$.
\end{The}

This result is analogous to theorem 1.8 in Aldous and Pitman \cite{ArtAldousPitman} and is proven by similar methods.

\begin{proof}
We need to couple two walks $Y$ and $\tilde{Y}$ that start at time $n_0$ in $s_0$ and $s$. Let again $V = Y - \tilde{Y}$ and denote $V_{n_0} = s_0 - s $ the difference. By Bezout's identity, lemma \ref{lem:Bezout}, and the fact that $\gcd(\mathfrak{D}) = 1$ we can write $1 = x_1 v_1 + \dots + x_p v_p$, hence $V_0 = V_0 x_1 d_1 + \dots + V_0 x_p d_p$.

\bigskip

The idea is that we use $p$ different versions of the Mineka coupling, one for each integer in $\mathfrak{D}$. First we use a coupling so that the difference of the walks makes step sizes $0, -d_1, d_1$. At a suitable time we switch to a coupling so that the difference makes steps of sizes $0,-d_2, d_2$ and so on. Define for $d \in \mathfrak{D}$
\begin{equation*}
\alpha_k^i(d) = \frac{1}{2} (\mu_{\lambda_i}(k) \wedge \mu_{\lambda_i}(k+d)).
\end{equation*}

We start the coupling just as in the case where $1 \in \mathfrak{D}$.
\begin{align*}
\PR[(R_i,\hat{R}_i) = (k-d_1,k)] & = \alpha^i_{k-1}(d_1) \\
\PR[(R_i,\hat{R}_i) = (k,k-d_1)] & = \alpha^i_{k-1}(d_1) \\
\PR[(R_i,\hat{R}_i) = (k,k)] & = p^i_k - \alpha^i_k(d_1) - \alpha^i_{k-1}(d_1)
\end{align*}  
Let $T_1 = \inf \{n \geq 0 \; : \; V_n = V_0 x_2 d_2 + \dots + V_0 x_p d_p  \}$. $T_1$ is almost surely finite because $d_1 \in \mathfrak{D}$ and the argument preceding lemma \ref{lem:demandcoupling}. Now let $Y_n = s_0 + \sum_{i= n_0 + 1}^n R_i$ and $\hat{Y}_n = s + \sum_{i=n_0 +  1}^n \hat{R}_i$ until $T_1$. 
From time $T_1$ onwards we let the walks evolve with steps of size $d_2$, so: 
\begin{align*}
\PR[(R_i,\hat{R}_i) = (k-d_2,k)] & = \alpha^i_{k-1}(d_2) \\
\PR[(R_i,\hat{R}_i) = (k,k-d_2)] & = \alpha^i_{k-1}(d_2) \\
\PR[(R_i,\hat{R}_i) = (k,k)] & = p^i_k - \alpha^i_k(d_2) - \alpha^i_{k-1}(d_2)
\end{align*}  
Conditional on its position at time $T_1$ we define $T_2 = \inf \{n \geq T_1 \; : \; V_n = V_0 x_3 d_3 + \dots + V_0 x_p d_p  \}$. Because $T_1$ is almost surely finite and the fact that $d_2 \in \mathfrak{D}$ it holds that $T_2$ is also almost surely finite.

\bigskip

We repeat the last step for $3$ to $p$ and obtain that $T_p = \inf \{n \geq T_{p-1} \; : \; V_n = 0  \}$ is almost surely finite. 
\end{proof}

This machinery is enough to start prove implication (3) to (1) in theorems \ref{The:equivA} and \ref{The:equivA2}. We start with theorem \ref{The:equivA}.

\begin{proof}[Proof of (3) to (1) of theorem \ref{The:equivA}]
Recall (3) of the theorem:
\begin{equation*}
\sum_{i: \lambda_i \geq 1} \frac{1}{\lambda_i} + \sum_{i : \lambda_i < 1} \lambda_i = \infty
\end{equation*}
First suppose that $\sum_{i : \lambda_i < 1} \lambda_i = \infty$. We check the condition in corollary \ref{Cor:basictrivialitydemand}.
\begin{align*}
& \sum_i \sum_k \PR[\eta_i = k] \wedge \PR[\eta_i = k+1] \\
& \quad \geq \sum_{i : \lambda_i < 1} \PR[\eta_i = 0] \frac{a_1}{a_0} \lambda_i \\
& \quad \geq \frac{1}{Z_1} \sum_{i : \lambda_i < 1} \lambda_i \\
& \quad = \infty
\end{align*}
Now suppose that $\sum_{i: \lambda_i \geq 1} \frac{1}{\lambda_i} = \infty$. We split this into two cases, let $B = \frac{a_{N-1}}{a_N} \vee 1$. 

First suppose that $\frac{a_{N-1}}{a_N} \geq 1$. If it is the case that 
\begin{equation} \label{eqn:inftysumB}
\sum_{i : 1 \leq \lambda_i \leq \frac{a_{N-1}}{a_N}} \frac{1}{\lambda_i} = \infty
\end{equation} 
then there are infinitely many $i$ such that $1 \leq \lambda_i < \frac{a_{N-1}}{a_N}$. For a given $k$ the set of probabilities
\begin{equation*}
\left\{ \PR[\eta_i = k] \; : \; i \text{ such that } 1 \leq \lambda_i < \frac{a_{N-1}}{a_N} \right\}
\end{equation*}
is bounded from below by some constant, hence $\sum_i \sum_k \PR[\eta_i = k] \wedge \PR[\eta_i = k+1] = \infty$.
The other case is that the sum for elements $i$ so that $\lambda_i > B$ is infinite:
\begin{equation*}
\sum_{i : \lambda_i \geq B} \frac{1}{\lambda_i} = \infty.
\end{equation*} 
For $\lambda_i > B$ a small calculation shows that $\PR[\eta_i = N]$ is bounded from below by $C = \frac{a_N B^N}{Z_B}$.
\begin{align*}
& \sum_i \sum_k \PR[\eta_i = k] \wedge \PR[\eta_i = k+1] \\
& \quad \geq \sum_{i : \lambda_i \geq B} \PR[\eta_i = N] \left(1 \wedge \frac{a_{N-1}}{a_N} \frac{1}{\lambda_i} \right) \\
& \quad = \sum_{i : \lambda_i \geq B} \PR[\eta_i = N] \frac{a_{N-1}}{a_N} \frac{1}{\lambda_i} \\
& \quad \geq C \frac{a_{N-1}}{a_N} \sum_{i : \lambda_i \geq B} \frac{1}{\lambda_i} \\
& \quad = \infty
\end{align*}
This proves theorem \ref{The:equivA} (3) to (1).
\end{proof}

Now we prove the second theorem.

\begin{proof}[Proof of (3) to (1) of theorem \ref{The:equivA2}]

For the proof of (1) given (3) we need $(\ast \ast)$, so either $\lambda^* < \infty$ or that $\lambda^* = \infty$ and there exists a finite set $\mathfrak{D} = \{d_1, \dots, d_p\}$ such that $\gcd(\mathfrak{D}) = 1$ and
\begin{equation*}
\mathfrak{D} \subset \left\{d \geq 1 \; : \; \sup_{k} \frac{a_k^2}{a_{k-d}a_{k+d}} = \sup_k \prod_{i=0}^{d-1}\frac{b(k+i+1,k-i-1)}{b(k-i,k+i)}< \infty \right\}
\end{equation*}

We take two approaches to show that this implies that $\mathcal{A}$ is trivial. The first approach resembles the proof of (3) to (1) of theorem \ref{The:equivA} above. Under the assumption that $\lambda^* < \infty$ this method will show that $\sum_i \lambda_i = \infty$ implies that $\mathcal{A}$ is trivial under $\PR = \mu_\lambda$. The second and more difficult approach will be used for the case that $\lambda^*$ is infinite.

\bigskip

We start the first approach by analysing the minimum of the two probabilities $\PR[\eta_i = k] \wedge \PR[\eta_i = k+1]$. Recall the definition of $I$ from assumption \ref{Ass:2}, then we see that
\begin{equation*}
\PR[\eta_i = k] \geq \frac{I}{\lambda_i} \PR[\eta_i = k+1].
\end{equation*}
Therefore:
\begin{align*}
\sum_i \sum_k \PR[\eta_i = k] \wedge \PR[\eta_i = k+1] & \geq \sum_i \sum_k \PR[\eta_i = k+1] \left(\frac{I}{\lambda_i} \wedge 1 \right) \\
& = \sum_i \left(\frac{I}{\lambda_i} \wedge 1 \right) \PR[\eta_i >0] \\
& = \sum_{i : \lambda_i < I} \PR[\eta_i >0] + I \sum_{i : \lambda_i \geq I} \frac{1}{\lambda_i} \PR[\eta_i >0]
\end{align*}
We know from assumption (2) that $\sum_i \PR[\eta_i > 0] = \infty$. Suppose now that the sum $\sum_{\lambda_i < I} \PR[\eta_i > 0] = \infty$, then we are done. On the other hand it could be that $\sum_{\lambda_i \geq I} \PR[\eta_i > 0] = \infty$.
Note the following for $\lambda_i \geq I$:
\begin{align*}
\PR[\eta_i >0 ] & = 1 - \PR[\eta_i = 0] \\
& = 1- \frac{1}{\lambda_i} \PR[\eta_i = 1] \\
& \geq 1- \frac{1}{\lambda_i} \PR[\eta_i >0]
\end{align*}
Hence
\begin{equation*}
\PR[\eta_i > 0] \geq \frac{\lambda_i}{1+ \lambda_i}
\end{equation*}
and note that $\frac{I}{1+I} \leq \frac{\lambda_i}{1+ \lambda_i} \leq 1$. This leads to the observation that
\begin{equation*}
\sum_{i : \lambda_i \geq I} \frac{1}{\lambda_i} \PR[\eta_i >0] \geq \frac{I}{1+I}\sum_{i : \lambda_i \geq I} \frac{1}{\lambda_i}
\end{equation*}
Because we know from the fact $\sum_{\lambda_i \geq I} \PR[\eta_i > 0] = \infty$ that there are infinitely many terms so that $\lambda_i \geq I$ we obtain that if $\sum_{i : \lambda_i \geq I} \frac{1}{\lambda_i} = \infty $
then
\begin{equation*}
\sum_i \sum_k \PR[\eta_i = k] \wedge \PR[\eta_i = k+1] = \infty.
\end{equation*}
Note that if $\lambda^* = \infty$ this is tricky to check. But suppose now that $\lambda^* < \infty$ then this is satisfied automatically. This is the first assumption of $(\ast \ast)$.
This however also proves the following:
If
\begin{equation} \label{eqn:Atrivialextraresult}
\sum_{i : \lambda_i < I } \lambda_i + \sum_{ i : \lambda_i \geq I} \frac{1}{\lambda_i} = \infty
\end{equation}
then $\mathcal{A}$ is trivial. We state this as a corollary after the proof of theorem \ref{The:equivA2}.

\bigskip

We now start with the second approach for the case where $\lambda^* = \infty$, $\liminf \lambda_i = \infty$ and $\sum_{i : \lambda_i \geq I} \frac{1}{\lambda_i} < \infty $. We will use the second part of $(\ast \ast)$ for this case.

\bigskip

Start by assuming the special case that $\mathfrak{D} =\{1\}$ i.e.
\begin{equation*}
\sup_k \frac{a_k^2}{a_{k-1}a_{k+1}} < \infty
\end{equation*}
the proof in full generality is only slightly more difficult but uses similar arguments.

\bigskip

Assume $\lambda^* = \infty$ Recall the definition of $a_k$ and note that $a_0 = a_1 = 1$ and define the following two sets.

\begin{equation}
\begin{aligned}
M^*(x) & =
\begin{cases}
\;\;  \left\{k \neq 0 \; : \; \frac{a_{k-1}}{a_k} < x < \frac{a_k}{a_{k+1}} \right\} \cup \{0\} & \mbox{if } x< 1 \\
\;\;  \left\{k \neq 0 \; : \; \frac{a_{k-1}}{a_k} < x < \frac{a_k}{a_{k+1}} \right\} & \mbox{if } x \geq 1
\end{cases} \\
M_*(x) & = \left\{k \neq 0 \; : \; \frac{a_{k-1}}{a_k} > x > \frac{a_k}{a_{k+1}} \right\}
\end{aligned}
\end{equation}
The intuition behind these sets is the following. Fix $x$. Let $k > 0$, then $k \in M^*(x)$ if $\mu_x(k) > \mu_x(k-1) \vee \mu_x(k+1)$. In the case that $k=0$ it holds that $0 \in M^*(x)$ if $\mu_x(0) > \mu_x(1)$. Hence in the sum $\sum_k \mu_x(k) \wedge \mu_x(k+1)$ the probability $\mu_x(k)$ is not present. On the other hand if $k \in M_*(x)$ the probability will occur twice. This means that
\begin{equation} \label{eqn:gebruikM*}
\sum_k \mu_x(k) \wedge \mu_x(k+1) = 1 - \sum_{k \in M^*(x)} \mu_x(k) + \sum_{k \in M_*(x)} \mu_x(k)
\end{equation}

\bigskip

We work as before, we have assumed that $\sum_i \lambda_i = \infty$, suppose that $\sum_{i : \lambda_i < 1} \lambda_i = \infty$, then we obtain that $\mathcal{A}$ is trivial under $\mu_\lambda$. This is because 
\begin{align*}
\sum_i \sum_k \PR[\eta_i = k] \wedge \PR[\eta_i =k+1] & \geq \sum_{i : \lambda_i < 1} \PR[\eta_i = 1] \\
& = \sum_{i : \lambda_i < 1} \frac{\lambda_i}{Z_{\lambda_i}} \\
& > \frac{1}{Z_1} \sum_{i : \lambda_i < 1} \lambda_i \\
& = \infty.
\end{align*}

This leaves the case that $\sum_{i : \lambda_i \geq 1} \lambda_i = \infty$. This tells us nothing but the fact that there are infinitely many sites $i$ so that $\lambda_i \geq 1$. We use this to prove divergence of 
\begin{equation*}
\sum_i \sum_k \PR[\eta_i = k] \wedge \PR[\eta_i = k+1].
\end{equation*}

From this point on if we we sum over $i$ we implicitly assume that $\lambda_i \geq 1$ to make the notation easier. The sets $M^*$ and $M_*$ allow us to rewrite this sum.

\begin{align*}
& \sum_i \sum_k \PR[\eta_i = k] \wedge \PR[\eta_i = k+1] \\
& \quad = \sum_i \left[ 1 - \sum_{k \in M^*(\lambda_i)} \mu_{\lambda_i}(k) + \sum_{k \in M_*(\lambda_i)} \mu_{\lambda_i}(k) \right] \\
& \quad \geq \sum_i \left[ 1 - \sum_{k \in M^*(\lambda_i)} \mu_{\lambda_i}(k) \right]
\end{align*}
So if we can bound $\sum_{k \in M^*(\lambda_i)} \mu_{\lambda_i}(k)$ away from $1$ uniformly in $i$, then this sum will be infinite.

\bigskip

For $k \in M^*(x)$ we see that 
\begin{equation*}
\frac{1}{x} < \frac{a_k}{a_{k-1}} = \frac{a_{k+1}}{a_k} \frac{a_k^2}{a_{k-1}a_{k+1}}
\end{equation*}
which in turn implies that
\begin{align*}
\mu_x(k) & = \frac{a_k x^k}{Z_x} \\
& = \frac{a_k x^{k+1}}{Z_x} \frac{1}{x} \\
& < \mu_x(k+1)\left[\frac{a_k^2}{a_{k-1}a_{k+1}} \right].
\end{align*}
We use this to bound the following conditional probability.
\begin{align*}
\frac{\mu_x(k)}{\mu_x(k) + \mu_x(k+1)} & < \frac{\mu_x(k+1 )}{\mu_x(k) + \mu_x(k+1)} \left[\frac{a_k^2}{a_{k-1}a_{k+1}} \right] \\
& = \left(1 - \frac{\mu_x(k)}{\mu_x(k) + \mu_x(k+1)} \right) \left[\frac{a_k^2}{a_{k-1}a_{k+1}} \right]
\end{align*}
This yields for $C(k) = \left[\frac{a_k^2}{a_{k-1}a_{k+1}} \right]$ that
\begin{equation} \label{eqn:fraction}
\frac{\mu_x(k)}{\mu_x(k) + \mu_x(k+1)} < \frac{C(k)}{1+C(k)}
\end{equation}
Note that bounding the sum $\sum_{k \in M^*(\lambda_i)} \mu_{\lambda_i}(k)$ away from $1$ uniformly in $i$ is possible by bounding away fractions of the type found in equation (\ref{eqn:fraction}) from $1$, but this is equivalent to bounding away $C(k)$ from $\infty$.
Hence we obtain as a condition that
\begin{equation*}
\sup_{x, k \in M^*(x)} \frac{a_k^2}{a_{k-1}a_{k+1}} < \infty
\end{equation*}
If we fix some $x$ and some $k \notin M^*(x)$ then this fraction is bounded by $1$, so it holds uniformly if
\begin{equation*}
\sup_k \frac{a_k^2}{a_{k-1}a_{k+1}} < \infty.
\end{equation*}

If we plug in theorem \ref{The:trivialitydemand} instead of corollary \ref{Cor:basictrivialitydemand} then we can improve on the last calculation to obtain (3) to (1) of theorem \ref{The:equivA2}.

\end{proof}

In the proof of theorem \ref{The:equivA2} we saw that the condition in equation (\ref{eqn:Atrivialextraresult}) is also enough to obtain triviality of $\mathcal{A}$.
\begin{Cor}
If
\begin{equation*} 
\sum_{i : \lambda_i < I } \lambda_i + \sum_{ i : \lambda_i \geq I} \frac{1}{\lambda_i} = \infty
\end{equation*}
then $\mathcal{A}$ is trivial under $\mu_\lambda$.
\end{Cor}

\section{A Product measure following a deterministic strictly increasing profile} \label{sec:deterministicprofile}

In the case that $W = \N$ one might think that $\sum_i \lambda_i = \infty$ is enough to prove that $\mathcal{A}$ is trivial. This is not the case as one can see in the next example. We construct a particle system and a non-ergodic stationary product measure, such that $\sum_i \lambda_i = \infty$. 

\bigskip

Construct a conservative product type particle system on $\N^{\N}$ with the following properties. First we define the nearest neighbour transition kernel $p$, which we condition to satisfy $p_{i,i+1} + p_{i,i-1} = 1$, let $p_{0,1}  = 1$ and let $p_{1,2} \in (0,1)$. First let $\lambda_i = (2i^2+1)!$ then define the other $p_{i,j} $ by 
\begin{equation} \label{eqn:exrever}
p_{i,i+1} = \frac{\lambda_{i} - \lambda_{i-1} p_{i-1,i}}{\lambda_i}
\end{equation}
By induction, using the fact that $\lambda_{i+1} > \lambda_i$ we see that indeed $p_{i,i+1} \in (0,1)$ for all $i$, so $p$ is irreducible. Furthermore from equation (\ref{eqn:exrever}) we see that $p$ is reversible with respect to $\lambda$, so this $\lambda$ is a candidate to generate an invariant measure for conservative particle systems. Now define the function $b$:
\begin{equation*}
b(n,k) = \frac{1}{ (2 (k+1))! }
\end{equation*}
and note that this function is bounded, so that there is a corresponding Markov process. Furthermore assumptions \ref{Ass:1} and \ref{Ass:2} are satisfied.

\bigskip

Define the product invariant measure $\PR = \mu_\lambda$ on $\N^\N$ associated to $\lambda$ by giving the coefficients $a_k$.
\begin{equation*}
a_k = \prod_{i=0}^k \frac{1}{(2i)!}
\end{equation*}
The one site marginal is as before $\PR[\eta_i = k] = a_k \lambda_i^k /Z_{\lambda_i}$. A small calculation shows that
\begin{align*}
\frac{\PR[\eta_k = k^2+n]}{\PR[\eta_k = k^2]} & < \frac{1}{(2k^2+2)^n} \\
\frac{\PR[\eta_k = k^2-n]}{\PR[\eta_k = k^2]} & < \frac{1}{(2k^2+1)^n} 
\end{align*}

Note that the sequence $\frac{a_{k+1}}{a_k}$ is decreasing. A moments thought shows that in this case corollary \ref{Cor:basictrivialitydemand} is in fact the strongest case of theorem \ref{The:trivialitydemand}.

Hence by equation (\ref{eqn:gebruikM*}) we see that
\begin{align*}
\sum_k \sum_n \PR[\eta_k = n] \wedge \PR[\eta_k = n+1] & = \sum_k 1 - \PR[\eta_k = k^2] \\
&= \sum_k \PR[\eta_k \neq k^2]
\end{align*}
So by Borel-Cantelli and corollary \ref{Cor:basictrivialitydemand} we see that $\PR[\eta_k \neq k^2 \text{ infinitely often}] = 1$ implies that $\mathcal{A}$ is trivial under $\mu_\lambda$. We now show that we can reverse the implication as well. Suppose that $\PR[\eta_k \neq k^2 \text{ infinitely often}] = 0$. Define $A_n = \{\eta \; : \; \sum_k (\eta_k - k^2) = n \}$ and note that $A_n \in \mathcal{A}$, furthermore a calculation shows that these sets have positive $\mu_\lambda$ measure, so indeed $\mathcal{A}$ is not trivial. So in this case we obtain a strengthening of corollary \ref{Cor:basictrivialitydemand} to
\begin{equation}
\mathcal{A} \text{ is trivial under } \mu \quad \Leftrightarrow \quad \PR[\eta_k \neq k^2 \text{ infinitely often}] = 1.
\end{equation}

We check which of the two possibilities is the case here.
\begin{align*}
& \sum_k 1 - \PR[\eta_k = k^2] \\
& \quad = \sum_k \left[\sum_{1 \leq n \leq k} \PR[\eta_k = k^2-n] + \sum_{n \geq 1} \PR[ \eta_k = k^2 +n]  \right] \\
& \quad \leq \sum_k \PR[\eta_k = k^2] \left[\sum_{1 \leq n \leq k} \frac{1}{(2k^2 + 1)^n}+ \sum_{n \geq 1} \frac{1}{(2k^2 + 2)^n} \right] \\
& \quad = \sum_k \PR[\eta_k = k^2] \left[\frac{1}{2k^2+1} \frac{2k^2+1}{2k^2} + \frac{1}{2k^2+2} \frac{2k^2+2}{2k^2+1} \right] \\
& \quad \leq 2 \sum_k \frac{1}{2k^2} \\
& \quad < \infty
\end{align*}
So we see that $\mathcal{A}$ is not trivial which in turn implies that $\mu_\lambda$ is not extremal and can be disintegrated into measures $\mu^{(n)}$ supported on the sets $A_n = \{\eta \; : \; \sum_k (\eta_k - k^2) = n \}$:
\begin{equation*}
\mu_\lambda(\cdot) = \sum_{n \in \Z} \mu^{(n)}(\cdot) \mu_\lambda(A_n)
\end{equation*}
where $\mu_\lambda(A_n) > 0$ for all $n$.

\section{Anti-particle perspective} \label{Sect:antiparticle}

In this section we elaborate on the symmetric nature of theorem \ref{The:ergodicproduct} (a). We work with a system where $L^{b,p}f(\eta) = \sum_{x,y} p(x,y) b(\eta_x,\eta_y) \nabla_{x,y}f(\eta)$ in the case that $W = \{0,\dots,N\}$. 

\bigskip

In the process particles jump from $x$ to $y$ with rate $p(x,y)b(\eta_x,\eta_y)$. Instead of saying that a particle jumps from $x$ to $y$ one could say that an anti-particle jumps from $y$ to $x$. The next step is to forget about the motion of the particles, and only look at the motion of the anti-particles. It is clear that this motion contains exactly the same information as the motion of the particles. 
This way of looking at the system has some consequences. Define the transformation $\theta$ and the following adapted versions of $b$ and $p$:
\begin{align*}
(\theta \eta)_x & = N - \eta_x \\ 
\tilde{b}(n,k) & = b(N-k,N-n) \\
\tilde{p}(x,y) & = p(y,x) \\
\end{align*}
The motion of the antiparticles can be described by these adapted versions. Clearly if there are $n$ particles at some site, then there are $N-n$ antiparticles, this explains the definition of $\theta$. A transition of a particle from $x$ to $y$ with rate $p(x,y)b(\eta_x,\eta_y)$ corresponds to the transition of an anti-particle from $y$ to $x$ with rate $p(x,y)b(\eta_x,\eta_y)$. We would like to rewrite this rate in such a way that we recognize it in a familiar form. First of all the anti-particle moves from $y$ to $x$ so we turn $p$ around: $\tilde{p}$. Also we must write $\eta_x$ and $\eta_y$ in anti-particle form, and we must write it in such a way that the first coordinate describes the number of anti-particles at the site of departure. This leads us to the rate of moving an anti-particle from $y$ to $x$ with rate $p(x,y)b(\eta_x,\eta_y) = \tilde{p}(y,x) \tilde{b}(N-\eta_y,N-\eta_x) = \tilde{p}(y,x) \tilde{b}((\theta \eta)_y,(\theta \eta)_x)$. Thus this model can be described by the generator
\begin{equation} \label{eqn:Lduality}
L^{\tilde{b},\tilde{p}}g(\theta \eta) = \sum_{x,y} \tilde{p}(x,y) \tilde{b}((\theta \eta)_x,(\theta \eta)_y) \nabla_{x,y} g(\theta \eta),
\end{equation}
which is exactly the form we are used to. Note that the assumptions that hold for the model associated with $L^{b,p}$ also hold for $L^{\tilde{b},\tilde{p}}$. Hence $L^{\tilde{b},\tilde{p}}$ generates a semigroup $\tilde{S}_t$. Without to much work we see that equation (\ref{eqn:Lduality}) leads to $L^{b,p}(f \circ \theta)(\eta) = L^{\tilde{b},\tilde{p}}f(\theta \eta)$.

\bigskip

We restate the results in theorem \ref{The:antiparticle}. 

\begin{The} \label{The:antiparticle}
For $f \in D$ it holds that
\begin{equation*}
L^{b,p}(f \circ \theta)(\eta) = L^{\tilde{b},\tilde{p}}f(\theta \eta).
\end{equation*}
For $f \in \overline{D} = C(E)$ it holds that
\begin{equation*}
S_t(f \circ \theta)(\eta) = \tilde{S}_tf(\theta \eta)
\end{equation*}
\end{The}

As a consequence we obtain that if $\{ \eta(t) \; : \; t \geq 0 \}$ is the Markov process generated by $L^{b,p}$, then $\{ (\theta \eta)(t) \; : \; t \geq 0 \}$ is the process with generator $L^{\tilde{b},\tilde{p}}$.

\subsection{Anti-particle perspective on product measures}

We look at how to interpret proposition \ref{prop:invariantmeasures} from this perspective. 

We start by looking at the case that $b$ is independent of the second variable, the case corresponding to proposition \ref{prop:invariantmeasures} (a) where $b(n,k) = b(n,0)$ for all $n$ and $k$. These $b$ are not suitable for the case that $W$ is a finite set, because we need the fact that $b(\cdot, N) = 0$.

\bigskip

For the cases (b) and (c) in proposition \ref{prop:invariantmeasures} the $\lambda$ that generate invariant measures satisfy one of the following two conditions.
\begin{itemize}
\item If $x \nsim y$ then $\lambda_x = \lambda_y$. 
\item For all $x$ and $y$ it holds that $x \thicksim y$.
\end{itemize}
In this case denote by $\pi = \frac{c}{\lambda}$ for some $c \geq 0$. Suppose that $\lambda$ solves 
\begin{equation*}
\sum_x \lambda_x p(x,y) = \lambda_y \sum_x  p(y,x).
\end{equation*}
Then it is an easy calculation that $\pi$ solves 
\begin{equation*}
\sum_x \pi_x \tilde{p}(x,y) = \pi_y \sum_x  \tilde{p}(y,x).
\end{equation*}
Furthermore if $\lambda$ is reversible with respect to $p$ then $\pi$ is reversible with respect to $\tilde{p}$ and the property that: if $x$ and $y$ are such that $\lambda_x p(x,y) \neq \lambda_y p(y,x)$, then $\lambda_x = \lambda_y$ is turned into: if $x$ and $y$ are such that $\pi_x \tilde{p}(x,y) \neq \pi_y \tilde{p}(y,x)$, then $\pi_x = \pi_y$.

\bigskip

We look at how the second item in proposition \ref{prop:invariantmeasures} should be interpreted for the anti-particle model. One would think that the anti-particle perspective would give invariant measures for functions $b$ so that $b(n,k) - b(k,n) = b(N,k) - b(N,n)$, but notice that a small calculation using assumption \ref{Ass:1} yields that $b(N,k) - b(N,n) = b(n,0) - b(k,0)$. 

The associated product measures must change too. Denote the values $\tilde{a}_n$ analogous to the values of $a_n$:
\begin{align*}
\tilde{a}_0 & = 1 \\
\tilde{a}_k & = \prod_{i=0}^{k-1} \frac{\tilde{b}(1,i)}{\tilde{b}(i+1,0)}  \\
\tilde{Z}_\pi & = \sum_k \pi^k \tilde{a}_k \\
\end{align*}
and define the product measure $\nu_\pi$ by its marginals $\nu_\pi(n) = \tilde{a}_n \pi_x^n \tilde{Z}^{-1}_{\pi_x}$. One would hope that an invariant product measure for $L^{b,p}$ given by $\mu_\lambda$ is also given by the invariant measure for $L^{\tilde{b},\tilde{p}}$ written by $\nu_\pi$ for some suitable $\pi$. This is indeed the case.

\begin{Prop} \label{prop:antiparticlemeaure}
Let $\pi = \frac{1}{\lambda} \frac{b(N,0)}{b(1,N-1)}$ then it holds that $\mu_\lambda(\dd \eta) = \nu_\pi(\dd \theta \eta)$.
\end{Prop}

\begin{proof}
We only need to prove the statement for a single marginal. So we assume that $\lambda$ and $\pi$ only have one index and prove that $\mu_\lambda(N-n) = \nu_\pi(n)$.
Recall assumption \ref{Ass:1} which we use to obtain the third line in the following computation.
\begin{align*}
\tilde{a}_n & = \prod_{i=0}^{n-1} \frac{\tilde{b}(1,i)}{\tilde{b}(i+1,0)} \\
& = \prod_{i=0}^{n-1} \frac{b(N-i,N-1)}{b(N,N-i-1)} \\
& = \prod_{i=0}^{n-1} \frac{b(N-i,0)}{b(1,N-i-1)} \frac{b(1,N-1)}{b(N,0)} \\
& = \left( \frac{b(1,N-1)}{b(N,0)} \right)^n \; \; \prod_{i=0}^{n-1} \frac{b(N-i,0)}{b(1,N-i-1)} \\
& = \left( \frac{b(1,N-1)}{b(N,0)} \right)^n \; \; \left[ \prod_{i=0}^{N-1} \frac{b(N-i,0)}{b(1,N-i-1)}\right] \left[ \prod_{i=n}^{N-1} \frac{b(N-i,0)}{b(1,N-i-1)} \right]^{-1} \\
& = \left( \frac{b(1,N-1)}{b(N,0)} \right)^n a_N^{-1} a_{N-n}
\end{align*}
It follows that
\begin{align*}
\tilde{a}_n \pi^n & = \left( \frac{b(1,N-1)}{b(N,0)} \right)^n a_N^{-1} a_{N-n} \left(\frac{1}{\lambda}\frac{b(N,0)}{b(1,N-1)} \right)^n   \\
& =   a_N^{-1} a_{N-n} \lambda^{N-n} \left(\frac{1}{\lambda} \right)^N 
\end{align*}
So if we use this information to calculate the probabilities we obtain
\begin{equation*}
\nu_\pi(n) = \frac{\tilde{a}_n \pi^n}{\tilde{Z}_\pi} = \frac{a_{N-n} \lambda^{N-n}}{Z_\lambda} = \mu_\lambda(N-n)
\end{equation*}
which is what we wanted to prove.
\end{proof}

Lastly we comment on the symmetric nature of theorem \ref{The:ergodicproduct} (a). Using $c = \frac{b(N,0)}{b(1,N-1)}$, the rewrite using the antiparticle model gives
\begin{equation*}
\sum_{i : \lambda_i < 1} \lambda_i + \sum_{i: \lambda_i \geq 1} \frac{1}{\lambda_i} = c \sum_{i : \pi_i > c} \frac{1}{\pi_i} + c^{-1} \sum_{i: \pi_i \leq c} \pi_i = \infty
\end{equation*}
which is equivalent to
\begin{equation*}
\sum_{i : \pi_i < 1} \pi_i + \sum_{i: \pi_i \geq 1} \frac{1}{\pi_i} = \infty
\end{equation*}
The sum over the $\lambda_i$ where $\lambda_i <1$ turns to the sum over the $\pi_i^{-1}$ where $\pi \geq 1$. The sum over the $\lambda_i^{-1}$ where $\lambda_i \geq 1$ turns to the sum over the $\pi_i$ where $\pi < 1$. 

\bigskip

\textbf{Acknowledgment}

The author would like to thank Frank Redig for fruitful discussions and for reading of and comments on the manuscript.

\bibliography{mybib}{}

\begin{thebibliography}{10}

\bibitem{ArtAldousPitman}
David Aldous and Jim Pitman.
\newblock On the zero-one law for exchangeable events.
\newblock {\em The Annals of Probability}, 7(4):704, 1979.

\bibitem{AnCoRoMisanthrope}
E.~Andjel, C.~Cocozza-Thivent, and M.~Roussignol.
\newblock Quelques compl\'ements sur le processus des misanthropes et le
  processus ``z\'ero-range''.
\newblock {\em Ann. Inst. H. Poincar\'e Probab. Statist.}, 21(4):363--382,
  1985.

\bibitem{AndjelZR}
Enrique~Daniel Andjel.
\newblock Invariant measures for the zero range process.
\newblock {\em The Annals of Probability}, 10(3):525, 1982.

\bibitem{LiggettBramson}
M.~Bramson and T.M. Liggett.
\newblock {Exclusion processes in higher dimensions: stationary measures and
  convergence.}
\newblock {\em Ann. Probab.}, 33(6):2255--2313, 2005.

\bibitem{CocozzaMisanthrope}
Christiane Cocozza-Thivent.
\newblock Processus des misanthropes.
\newblock {\em Probability Theory and Related Fields}, 70:509--523, 1985.
\newblock 10.1007/BF00531864.

\bibitem{EthierKurtz}
Stewart~N. Ethier and Thomas~G. Kurtz.
\newblock {\em Markov processes: Characterization and Convergence}.
\newblock Wiley, 1986.

\bibitem{GroRedVafInc}
Stefan Grosskinsky, Frank Redig, and Kiamars Vafayi.
\newblock Condensation in the inclusion process and related models.
\newblock {\em Journal of Statistical Physics}, 142:952--974, 2011.
\newblock 10.1007/s10955-011-0151-9.

\bibitem{IosifescuTail}
Marius Iosifescu.
\newblock On finite tail $\sigma$-algebras.
\newblock {\em Probability Theory and Related Fields}, 24:159--166, 1972.

\bibitem{JonesJonesNumberTheory}
Gareth~A. Jones and J.~Mary Jones.
\newblock {\em Elementary number theory}.
\newblock Springer-Verlag, 1998.

\bibitem{Jung}
Paul Jung.
\newblock Extremal reversible measures for the exclusion process.
\newblock {\em Journal of Statistical Physics}, 112:165--191, 2003.
\newblock 10.1023/A:1023679620839.

\bibitem{Kallenberg}
Olav Kallenberg.
\newblock {\em Foundations of Modern Probability}.
\newblock Springer-Verlag, second edition, 2002.

\bibitem{BookLiggett1}
Thomas~M. Liggett.
\newblock {\em Interacting Particle Systems}.
\newblock Springer-Verlag, 1985.

\bibitem{BookLiggett2}
Thomas~M. Liggett.
\newblock {\em Stochastic Interacting Systems: Contact, Voter and Exclusion
  Processes}.
\newblock Springer-Verlag, 1999.

\bibitem{LindvallCoupling}
Torgny Lindvall.
\newblock {\em Lectures on the coupling method}.
\newblock Dover Publications, Inc., 1992.

\bibitem{ArtSaadaLTSEP}
Ellen Saada.
\newblock A limit theorem for the position of a tagged particle in a simple
  exclusion process.
\newblock {\em The Annals of Probability}, 15(1):375, 1987.

\bibitem{ArtSaadaZRtaggedparticle}
Ellen Saada.
\newblock Processus de zero-range avec particule marquee.
\newblock {\em Annales de l'Institut Henri Poincar\`{e}, B}, 26(1):5, 1990.

\bibitem{SSEx}
Sunder Sethuraman.
\newblock On extremal measures for conservative particle systems.
\newblock {\em Ann. Inst. H. Poincare, Probabilites et Statistiques},
  37(2):139, 2001.

\bibitem{Spitzer}
Frank Spitzer.
\newblock {\em Principles of Random Walk}.
\newblock Springer-Verlag, 1976.

\bibitem{Thorisson}
Hermann Thorisson.
\newblock {\em Coupling, stationarity, and regeneration}.
\newblock Springer-Verlag, 2000.

\bibitem{UspenskyHeasletElNumberTheory}
J.V. Uspensky and M.A. Heaslet.
\newblock {\em Elementary number Theory}.
\newblock McGraw-Hill Book Company Inc., 1939.

\end{thebibliography}
\bibliographystyle{plain}

\end{document}